\documentclass[11pt,a4paper]{article}

\usepackage[T1]{fontenc}
\usepackage{amsmath,amsthm,amsopn,amsfonts,amssymb,dsfont,color}
\usepackage{mathrsfs}
\usepackage{tikz}
\usepackage{statex}
\usepackage{xcolor}
\usepackage{mathrsfs}
\usepackage{mathtools}
\usepackage{accents}
\usepackage{setspace}
\usepackage{appendix}
\usepackage{times}
\usepackage{titlesec,float}
\usepackage{xcolor}
\usepackage{listings}
\usepackage{enumerate}
\usepackage{indentfirst}
\usepackage{a4,a4wide}
\usepackage{tocloft}
\usepackage{tikz}
\usepackage[marginal]{footmisc}

\newtheorem{theorem}{\textbf{Theorem}}[section]
\newtheorem{lemma}[theorem]{\textbf{Lemma}}
\newtheorem{proposition}[theorem]{\textbf{Proposition}}
\newtheorem{claim}[theorem]{\textbf{Claim}}

\newtheorem{remark}[theorem]{\textbf{Remark}}

\numberwithin{equation}{section}

\makeatletter
\g@addto@macro\th@plain{\thm@headpunct{}}
\makeatother

\newcommand\bea{\begin{eqnarray}}
\newcommand\eea{\end{eqnarray}}
\newcommand\beaa{\begin{eqnarray*}}
\newcommand\eeaa{\end{eqnarray*}}

\title{Sharp estimates for the spreading speeds of the Lotka-Volterra competition-diffusion system: the strong-weak type}
\date{}

\begin{document}

\maketitle

\begin{center}
{\large\bf
Chang-Hong Wu
\footnote{Department of Applied Mathematics, National Yang Ming Chiao Tung University, Hsinchu, Taiwan.

e-mail: {\tt changhong@math.nctu.edu.tw}},
Dongyuan Xiao\footnote{IMAG, Univ. Montpellier, CNRS, Montpellier, France.

e-mail: {\tt dongyuan.xiao@umontpellier.fr}} and
Maolin Zhou\footnote{Chern Institute of Mathematics and LPMC, Nankai University, Tianjin, China.

e-mail: {\tt zhouml123@nankai.edu.cn}}
} \\
[2ex]
\end{center}


\tableofcontents

\vspace{10pt}

\begin{abstract}
We consider the classical two-species Lotka-Volterra competition-diffusion system in the strong-weak competition case.
When the corresponding minimal speed of the traveling waves is not linear determined,
we establish the precise asymptotic behavior of the solution of the Cauchy problem
in two different situations: (i) one species is an invasive one and the other is a native species;
(ii) both two species are invasive species.
\\


\end{abstract}

\section{Introduction}\label{s:intro}

We consider the following two-species Lotka-Volterra competition-diffusion system
\begin{equation}\label{system}
\left\{
\begin{aligned}
&\partial_tu=u_{xx}+u(1-u-av), & t>0,\ x\in\mathbb{R},\\
&\partial_tv=dv_{xx}+rv(1-v-bu), & t>0,\ x\in \mathbb{R},
\end{aligned}
\right.
\end{equation}
where $u=u(t,x)$ and $v=v(t,x)$ represent the population densities of two competing species at the time $t$ and position $x$. Here, all parameters are assumed to be positive: $d$ and $r$ stand for the diffusion rate and intrinsic growth rate of $v$, respectively; $a$ and $b$ represent the competition coefficient of $v$ and $u$, respectively. In present paper, we focus on the strong-weak competition case:
\begin{itemize}
\item[{\bf(H1)}] $0<a<1<b$.
\end{itemize}
The condition {\bf(H1)} indicates that species $u$ is the superior species; while $v$ is an inferior one.

Early in 1937, Fisher \cite{Fisher} and Kolmogorov, Petrovsky, and Piskunov \cite{KPP} introduced
a scalar reaction-diffusion equation $w_t=w_{xx}+f(w)$
with monostable nonlinearity

$$f'(0)>0,\ f'(1)<0,\ f(0)=f(1)=0,\ f(w)>0\ \text{for all}\ (0,1),$$
to model the propagation of dominant gene in a homogeneous environment. With the so-called KPP condition
$f'(0)w\geq f(w)$ for all $w\in[0,1]$, they proved that traveling wave solutions of the form $w(t, x)=W(x-ct)$
connecting the states  $1$  and $0$ exist if and only if
$c\geq c_{\min}:=2\sqrt{f'(0)}$, where $c_{\min}$ is called the minimal wave speed.
Moreover, they found a mathematical approach to describe the propagation of dominant gene by studying the long-time behavior of the solution to the Fisher-KPP equation:
\begin{equation}\label{single kpp}
\left\{
\begin{aligned}
&w_t=w_{xx}+f(w),\ \ t>0,\ x\in\mathbb{R},\\
&w(0,x)=w_0(x),\ \ x\in\mathbb{R},
\end{aligned}
\right.
\end{equation}
where $w_0(x):=H(-x)$, and $H(x)$ is the Heaviside function.

In the case that $w_0(x)\not\equiv 0$ is a nonnegative compactly supported function, 
Aronson and Weinberger \cite{Aronson Weinberger} showed that there exists a unique speed $c_w$ such that the solution of (\ref{single kpp}) satisfies
\beaa
\lim_{t\to\infty}\sup_{|x|\ge ct}w(t,x)=0\ \ \mbox{for all} \ \ c>c_w\quad\text{and}\quad
\lim_{t\to\infty}\sup_{|x|\le ct}|1-w(t,x)|=0\ \ \mbox{for all} \ \ c<c_w.
\eeaa
Moreover, the spreading speed $c_w$ coincides with the minimal wave speed $c_{\min}$.
The propagation phenomenon and inside dynamics
of the front for more general scalar equation have been discussed widely in the literature.
We may refer to, e.g., \cite{Berestycki Hamel2012,Bramson1983,Fife McLeod1977, Gardner et al 2012,log delay 3, Liang Zhao 2007,Roques et al 2012, Rothe1981,Stokes1976,Uchiyama1978} and references cited therein.

To understand the long-time behavior of solutions of system \eqref{system}, traveling wave solutions play an important role.
In the absence of the one species, namely $u$ or $v$,  system \eqref{system} can be reduced to the single scale Fisher-KPP equation like \eqref{single kpp}, which admits a unique (up to translations) traveling wave solution $U_{KPP}(x-ct)$ (resp. $V_{KPP}(x-ct)$) with the minimal speed
\beaa
c_u:=2,\quad(resp.\ \ c_v:=2\sqrt{rd}).
\eeaa
Regarding the traveling wave solutions of system \eqref{system} with {\bf(H1)}, Kon-on  \cite{Kan-on1997} showed
that there exists a unique speed $c^*\in[2\sqrt{1-a},2]$ 
such that system \eqref{system}
admits a solution $(c,U,V)$
satisfying
\begin{equation}\label{tw solution}
\left\{
\begin{aligned}
&U''+cU'+U(1-U-aV)=0,\\
&dV''+cV'+rV(1-V-bU)=0,\\
&(U,V)(-\infty)=(1,0),\ (U,V)(\infty)=(0,1),\\
&U'<0,\ V'>0,
\end{aligned}
\right.
\end{equation}
if and only if $c\ge c^*$. Thus, $c^*$ is called the minimal traveling
wave speed of system \eqref{tw solution}.

The linear determinacy of $c^*$ has been widely discussed over several decades to understand
the dynamics of diversity for invasive species.
It is said that linear determinacy holds if $c^*=2\sqrt{1-a}$ since the linearization of \eqref{tw solution} at the unstable state $(0,1)$ results in the linear speed $2\sqrt{1-a}$
(see \cite{Lewis Li Weinberger 1, Weinberger Lewis Li 2002}). In this case, $c^*$ is also said to be linear or be linearly selected. If $c^*>2\sqrt{1-a}$, we say that linear determinacy does not hold, nonlinear determinacy holds, or $c^*$ is nonlinear selected. Another terminology comes from Stokes \cite{Stokes1976}. We may say that $c^*=2\sqrt{1-a}$ is "pulled fronts" case since the propagation speed is determined only by the leading edge of the distribution of the population; while $c^*>2\sqrt{1-a}$
called "pushed fronts" case since the propagation speed is not determined by the behavior of the leading edge of the population distribution, but by the whole wavefront.
We also refer to the work of Roques et al. \cite{Roques et al 2015} that introduced another definition of pulled and pushed fronts for system \eqref{system}.
In present paper, we mainly focus on the "pushed fronts" case:
\begin{itemize}
\item[{\bf(H2)}] $c^*>\sqrt{1-a}$.
\end{itemize}

Linear/nonlinear determinacy of the minimal traveling wave speed of system \eqref{tw solution} has been investigated
in the literature.
Among them, Lewis, Li and Weinberger \cite{Lewis Li Weinberger 1} showed that linear determinacy holds when
\bea\label{LLW-cond}
0<d<2\quad {\rm and}\quad  r(ab-1)\leq (2-d)(1-a).
\eea
An improvement for the sufficient condition for linear determinacy were made by Huang \cite{Huang2010}:
\begin{equation}\label{huang-condition}
\frac{(2-d)(1-a)+r}{rb}\ge \max\Big\{a,\frac{d-2}{2|d-1|}\Big\}.
\end{equation}
Note that \eqref{LLW-cond} and \eqref{huang-condition} are equivalent when $d\leq 2$.
Roques et al. \cite{Roques et al 2015} numerically suggested that the parameters region for linear determinacy
can still be improved. More recently, Alhasanat and Ou \cite{Alhasanat Ou2019} made some improvements.

For $c^*$ being nonlinear selected, Huang and Han \cite{HuangHan2011} constructed examples in which linear determinacy fails to hold under the conditions: $r=d$ and $a$ is sufficiently close to $1$. Alhasanat and Ou \cite{Alhasanat Ou2019}
proved that $c^*$ is nonlinear if
\beaa
\frac{(d+2)(1-a)+r}{rb}<1-2(1-a).
\eeaa
Therefore, the assumption {\bf(H2)} is not void.
For related discussions, we also refer to, e.g.,
\cite{Alhasanat Ou2019-1, Guo Liang, Holzer Scheel 2012, Hosono 1998, Hosono 2003} and the references cited therein.

For the "pulled fronts" case $c^*=2\sqrt{1-a}$,
the long-time behavior of the solution of system \eqref{system}  is more complicated.
We strongly believe that logarithmic phase drift of the location of the wavefront exists as what happens for the scalar monostable equation. This problem will be discussed in our forthcoming paper.

\subsection{Main results}

The purpose of this paper is to establish
the sharp estimate on the long-time behavior of the solution of system \eqref{system} in the "pushed fronts" case {\bf(H2)}
with two different scenarios for initial data
$(u_0,v_0)$:

\begin{equation}\label{initial data}
\begin{aligned}
&u(0,x)=u_0(x)\in C(\mathbb{R},[0,1])\setminus\{0\}:\ \mbox{with compact support},\\
&v(0,x)=v_0(x)\in C(\mathbb{R},[0,1])
:\ \mbox{with a positive lower bound,}
\end{aligned}
\end{equation}
or
\begin{equation}\label{initial data2}
\begin{aligned}
(u,v)(0,x)=(u_0,v_0)(x)\in [C(\mathbb{R},[0,1])\setminus\{0\}]^2:\ \mbox{both are with compact support}.
\end{aligned}
\end{equation}
Biologically, \eqref{initial data} means that the species $u$ is the invasive species, while $v$ is the native species occupying the whole space;
\eqref{initial data2} indicates that both two species are invasive species.

There is a wide variety of literature regarding 
the traveling wave solution and (asymptotic) spreading speeds for system (\ref{system}).
In the weak competition case (i.e., $a<1,\ b<1$),
Tang and Fife \cite{Tang Fife 1980} established the existence of the minimal wave speed for traveling waves connecting
(0,0) and the coexistence state. For the Cauchy problem,
Lin and Li \cite{Lin Li} considered system (\ref{system}) with compactly supported initial functions and obtained the spreading speed of the faster species
and some estimates on the speed of the slower species. More recently, Liu, Liu and Lam \cite{Liu Liu Lam 1, Liu Liu Lam 2} obtained  rather complete results.

In the strong (bistable) competition case (i.e., $a>1,\ b>1$),
the existence of traveling waves connecting $(0,1)$ and $(1,0)$ was established by Gardner \cite{Gardner}, Conley and Gardner \cite{Conley Gardner 1984}
and Kan-on \cite{Kan-On}. For the Cauchy problem,
Carrere \cite{Carrere} studied the asymptotic spreading speed of the solution with initial data which are absent on the right half-line $x>0$ , and the slower species dominates the faster one on the left half-line $x<0$. More recently, Peng, Wu and Zhou \cite{Peng Wu Zhou} provided rigorous estimates on the spreading speed and profiles of the solution as $t\to\infty$.

In the critical competition case (i.e., $a=b=1,$), Alfaro and Xiao \cite{Alfaro Xiao} proved the non-existence of traveling waves with some monotonicity. Moreover, they studied the large time behavior of the solution of the Cauchy problem with compactly supported initial data. More precisely, they not only reveal that the "faster" species excludes the "slower" one, but also found a new bump phenomenon which provides a sharp description of the profile of the solution.

Regarding the strong-weak (monostable) competition case (i.e., {\bf(H1)}), the asymptotic spreading speed of the Cauchy problem
was firstly studied by Lewis, Li and Weinberger \cite{Lewis Li Weinberger 1, Lewis Li Weinberger 2} with $(u_0,v_0)$ satisfying
$0\leq u_0(x)\leq 1$ and $0\leq 1-v_0(x)\leq 1$, and both $u_0$ and $1-v_0$ are compactly supported functions.
Recently,  Girardin and Lam \cite{Girardin Lam} studied the spreading speed of the Cauchy problem
with initial data that are null or exponentially decay on the right half line. They obtained a complete understanding of the spreading properties by constructing very technical super-solutions and sub-solutions. Among other things, they also found that a so-called "nonlocal pulling" phenomenon may happen in some cases.

Regarding the study of the spreading property for other reaction-diffusion systems,
we refer to \cite{Iida Lui Ninomiya,Lui, Roquejoffre Terman Volpert} for monotone systems;
\cite{Ducrot Giletti Matano, Mori Xiao} for non-cooperative systems.


Our first result considers the scenario that the initial data $(u_0,v_0)$ satisfies \eqref{initial data}.
The spreading speed has been obtained in \cite{Lewis Li Weinberger 1}.
Here we establish the sharp long-time behavior of the solution when linear determinacy does not hold.


\begin{theorem}\label{th: profile}
Assume that {\bf(H1)}-{\bf(H2)} hold. Then
the solution $(u,v)$ of system (\ref{system}) with initial data (\ref{initial data}) satisfies
\begin{equation*}\label{profile of the solution cv>cu for small x}
\lim_{t\to\infty}\Big(\sup_{x\in[0,\infty)} \big|u(t,x)-U(x-c^*t-h)\big|+\sup_{x\in [0,\infty)}\big|v(t,x)-V(x-c^*t-h)\big|\Big)=0,
\end{equation*}
where $h$ is a constant and $(c^*,U,V)$ is the minimal traveling wave defined as \eqref{tw solution}.
\end{theorem}

Next, we consider the scenario that the initial function $(u_0,v_0)$ satisfies \eqref{initial data2}.
In this case, the spreading property becomes more complicated:
the invading speed of the stronger species $u$ could be nonlocal determined in some cases,
 as reported in \cite{Girardin Lam}. To give a precise illustration, let us
recall the auxiliary function given in \cite{Girardin Lam}:
\begin{equation}\label{f}
f(c):=c-\sqrt{c^2-4(1-a)}+2\sqrt{a}\ \ \mbox{and}\ \ f^{-1}(c'):=\frac{c'}{2}-\sqrt{a}+\frac{2(1-a)}{c'-2\sqrt{a}}.
\end{equation}
Note that $f$ is a decreasing function. If $2\sqrt{rd}\in(2,f(c^*))$, then we define the accelerated speed
\begin{equation*}
c_{**}:=f^{-1}(2\sqrt{rd})=\sqrt{rd}-\sqrt{a}+\frac{1-a}{\sqrt{rd}-\sqrt{a}}\in(c^*,2).
\end{equation*}
It has been showed in \cite{Girardin Lam} that, if $c_v>c_u$, there exist two wavefronts. The fast one moves with the speed $c_v$. The slow one moves with the speed $\mathscr{C}$, which satisfies
\begin{equation}\label{definition of bar c}
\left\{
\begin{aligned}
&\mathscr{C}=c^*\ \ \mbox{if}\ \ c_v\in [f(c^*),\infty),\\
&\mathscr{C}=c_{**}\ \ \mbox{if}\ \ c_v\in (2,f(c^*)).
\end{aligned}
\right.
\end{equation}

In \cite{Girardin Lam}, they found this "nonlocal pulling" phenomenon from an observation on the behavior of the solution $(u,v)$ on the leading edge, namely the region where $u,v\approx 0$.
Let us define functions for $c\ge c^*$ and $c'\ge\max\{c,f(c)\}$ as follows:
\begin{equation*}
\Lambda(c,c'):=\frac{1}{2}\Big(c'-\sqrt{c'^2-4\lambda(c)(c'-c)-4}\Big)\ \ \mbox{with}\ \ \lambda(c)=\frac{1}{2}\big(c-\sqrt{c^2-4(1-a)}\big).
\end{equation*}
If $c_v>c_u$, then we can assume $v$ invades the uninhabited region ($u,v\approx 0$) with a speed $c_1\ge c_v$ and $u$ chase $v$ from behind with a speed $c_2\in[c^*,c_u]$. In the region where $v\approx 1$, the profile of $u$ converges to the traveling wave solution defined as \eqref{tw solution} with speed $c_2$. Therefore, we have
$u(t,x)\approx e^{-\lambda(c_2)(x-c_2t)}$. Define a new function $y(t,x)=u(t,x)e^{\lambda(c_2)(x-c_2t)}$. In the range $x=\tilde{c}t$ with $\tilde{c}>c_1$, where $u,v\approx 0$, it holds
\begin{equation}\label{equation of y}
\partial_t y-y_{xx}\approx(1+\lambda(c_2))(\tilde{c}-c_2)y.
\end{equation}
Then, by assuming the exponential ansatz $y(t,x)\approx e^{-\Lambda(x-\tilde{c}t)}$, \eqref{equation of y} leads to the equation
$$\Lambda^2-\tilde{c}\Lambda+(1+\lambda(c_2)(\tilde{c}-c_2))=0.$$
The minimal root of the this equation is equal to $\Lambda(c_2,\tilde{c})$, which exists if and only if
$$\tilde{c}^2-4(\lambda(c_2)(\tilde{c}-c_2)+1)\ge 0.$$
This inequality immediately implies that $\tilde{c}$ has to satisfy $\tilde{c}\ge f(c_2)$, which implies $c_1\ge f(c_2)$.

More precisely,
we have the following propagation properties:
\begin{proposition}[Theorem 1.1 in \cite{Girardin Lam}]\label{prop: spreading speed}
Let $(u,v)$ be the solution of system \eqref{system} with initial data
$u_0\in C(\mathbb{R},[0,1])\setminus\{0\}$
with support included in a left half-line and $v_0\in C(\mathbb{R},[0,1])\setminus\{0\}$ with compact support.
Then the following hold:
\begin{itemize}
	\item[(1)] If $c_u>c_v$, then it holds
	\begin{equation*}
	\lim_{t\to\infty}\Big(\sup_{x\in[0,\infty)}v(t,x)+\sup_{x\ge c_1t}u(t,x)+\sup_{0\le x\le c_2t}|1-u(t,x)|\Big)=0
	\end{equation*}
	for all $0<c_2<c_u<c_1$.
	\item[(2)] If $c_u<c_v$, then it holds
	\begin{equation*}
	\lim_{t\to\infty}\sup_{x\ge c_1t}\Big(u(t,x)+v(t,x)\Big)=0\quad\mbox{for all}\quad c_1>c_v;
	\end{equation*}
	\begin{equation*}
	\lim_{t\to\infty}\sup_{c_2t\ge x\ge c_3t}\Big(u(t,x)+|1-v(t,x)|\Big)=0\quad \mbox{for all}\quad c_v>c_2>c_3>\mathscr{C};
	\end{equation*}
	\begin{equation*}
	\lim_{t\to\infty}\sup_{c_4t\ge x\ge 0}\Big(|1-u(t,x)|+v(t,x)\Big)=0\quad \mbox{for all}\quad c_4<\mathscr{C}.
	\end{equation*}
\end{itemize}
\end{proposition}

Here we first establish the convergence of the solution to system \eqref{system} with initial data
\eqref{initial data2}. For $c_u>c_v$, in view of statement (1) in Proposition~\ref{prop: spreading speed},
we see that $u$ is the only survival specie, so it can be seen as the fastest species.
Therefore, Corollary~4.6 in \cite{Peng Wu Zhou} can be applied to
obtain the propagating behavior of $u$ over $\{x\geq (c_v+\varepsilon) t\}$ for all $\varepsilon>0$ and large $t$.
Thus, combining Proposition~\ref{prop: spreading speed}(i) and \cite[Corollary~4.6]{Peng Wu Zhou},
we immediately conclude that

\begin{proposition}
Assume that {\bf(H1)} holds.
If $c_u>c_v$, then the solution $(u,v)$ of system (\ref{system}) with initial data (\ref{initial data2}) satisfies
\beaa
\lim_{t\to\infty}\Big[\sup_{x\in[0,\infty)}
\Big|u(t,x)-U_{KPP}\Big(x-c_{u} t+{\frac{3}{c_u}}\ln t+\omega(t)\Big)\Big|+\sup_{x\in[0,\infty)}|v(t,x)|\Big]=0,
\eeaa
where $\omega$ is a bounded function defined on $[0,\infty)$.
\end{proposition}

For $c_u<c_v$, we shall establish the following result.

\begin{theorem}\label{th2: profile}
Assume that {\bf(H1)}-{\bf(H2)} hold.
Let $f$ be the auxiliary function defined in \eqref{f}.
If $c_v\in [f(c^*),\infty)$, then the solution $(u,v)$ of system (\ref{system}) with initial data (\ref{initial data2}) satisfies
\beaa
&&\lim_{t\to\infty}\left[\sup_{x\in[c_0t,\infty)}\Big|v(t,x)-V_{KPP}\Big(x-c_{v} t+\frac{3d}{c_v}\ln t+\omega(t)\Big)\Big|+\sup_{x\in[c_0t,\infty)}|u(t,x)|\right]=0
\eeaa
and
\beaa
&&\lim_{t\to\infty}\left[\sup_{x\in[0,c_0t)}
\Big|u(t,x)-U(x-c^* t-\hat{h})\Big|+\sup_{x\in[0,c_0t)}\Big|v(t,x)-V(x-c^* t-\hat{h})\Big|
\right]=0,
\eeaa
where $c_0\in(c^*, c^*+\varepsilon)$ and $\hat{h}\in\mathbb{R}$ are some constants,
$\omega(\cdot)$ is a bounded function defined on $[0,\infty)$, and $(c^*,U,V)$ is the minimal traveling wave defined as \eqref{tw solution}.
\end{theorem}

\begin{remark}
Theorem~\ref{th2: profile} shows the sharp estimate on the long-time behavior of the solution reported in statement (2) of
Proposition~\ref{prop: spreading speed}  when the "nonlocal pulling" phenomenon does not occur, i.e., $\mathscr{C}=c^*$.
The case that $\mathscr{C}=c_{**}$ is more challenging, and will be discussed in our future work.
\end{remark}

Next, we recall some useful known results and establish some asymptotic estimates of the traveling wave for later use.

\subsection{Preliminaries}

\subsubsection{Comparison principle}
For the reader's convenience, we first recall the definitions of super-solution and sub-solution, and the comparison principle. Readers also can see section 2.1 of \cite{Girardin Lam} to find more details.
Define the operators as follow:
\begin{equation*}
N_1[u,v](t,x):=u_t-u_{xx}-F(u,v),\quad N_2[u,v](t,x):=v_t-dv_{xx}-G(u,v),
\end{equation*}
where
\bea\label{FG}
F(u,v):=u(1-u-av),\quad G(u,v):=rv(1-v-bu)
\eea
We say that $(\bar{u},\underline{v})\in[C(\bar{\Omega})\cap C^{2,1}(\Omega)]^2$ is a pair of super-solution (sub-solution) of system (\ref{system}) in
$$\Omega:=(t_1,t_2)\times(x_1,x_2),\ 0\le t_1<t_2\le \infty,\ -\infty\le x_1<x_2\le+\infty$$
if $(\bar{u},\underline{v})$ satisfies $N_1[\bar{u},\underline{v}]\ge 0$ and $N_2[\bar{u},\underline{v}]\le 0$ ($N_1[\bar{u},\underline{v}]\le 0$ and $N_2[\bar{u},\underline{v}]\ge 0$) in $\Omega$.

\begin{proposition}(Comparison Principle)\label{prop: cp}
Let $(\bar{u},\underline{v})$ and $(\underline{u},\bar{v})$ be a super-solution and sub-solution of system (\ref{system}) in $\Omega$, respectively. If $(\bar{u},\underline{v})$ and $(\underline{u},\bar{v})$ satisfy
\begin{equation}\label{cp boundary condition}
\left\{
\begin{aligned}
&\bar{u}(t_1,x)\ge \underline{u}(t_1,x),\ \underline{v}(t_1,x)\le\bar{v}(t_1,x)\quad\mbox{for all}\quad x\in(x_1,x_2),\\
&\bar{u}(t,x_i)\ge \underline{u}(t,x_i),\ \underline{v}(t,x_i)\le\bar{v}(t,x_i)\quad\mbox{for all}\quad t\in(t_1,t_2)\  \ \mbox{and}\ \ i=1,2,
\end{aligned}
\right.
\end{equation}
then it holds $\bar{u}\ge\underline{u}$ and $\underline{v}\le\bar{v}$ in $\Omega$. If $x_1=-\infty$ or $x_2=+\infty$, the corresponding boundary condition can be omitted.
\end{proposition}

\begin{remark}\label{weak supersol}
If both $(\bar{u}_1,\underline{v})$ and $(\bar{u}_2,\underline{v})$ are super-solutions, then
Proposition~\ref{cp boundary condition} still holds if $(\bar{u},\underline{v})$ is replaced by
$(\min\{\bar{u}_1,\bar{u}_2\},\underline{v})$. If  $(\bar{u},\underline{v}_1)$ and $(\bar{u},\underline{v}_2)$ are super-solutions, then
Proposition~\ref{cp boundary condition} still holds if $(\bar{u},\underline{v})$ is replaced by $(\bar{u},\max\{\underline{v}_1, \underline{v}_2\})$. More details can be found in \cite[Section 2]{Girardin Lam}.
\end{remark}

\subsubsection{Asymptotic behavior of the minimal traveling wave near $\pm\infty$}\label{subsec:Asymp behavior}

The asymptotic behavior of traveling waves $(c,U,V)$ near $\pm\infty$ for any $c\geq c^*$ has been reported
in \cite{MoritaTachibana2009} (see also \cite{Girardin Lam}).
In this subsection, we only recall those for $c=c^*>2\sqrt{1-a}$.


Let $(c,U,V)$ be a solution of system (\ref{tw solution}). To describe the asymptotic behavior of $(U,V)$ near $+\infty$,
we define
\beaa
&&\lambda_u^{\pm}(c):=\frac{-c\pm\sqrt{c^2-4(1-a)}}{2}<0,\\ &&\lambda_v^{-}(c):=\frac{-c-\sqrt{c^2+4rd}}{2d}<0<\lambda_v^{+}(c):=\frac{-c+\sqrt{c^2+4rd}}{2d}.
\eeaa

\begin{lemma}[\cite{MoritaTachibana2009}]\label{lm: behavior around + infty}
Let $(c^*,U,V)$ be the minimal traveling wave of system (\ref{tw solution}) with $c^*$ satisfying {\bf(H2)}
.
Then there exist positive constants $l_i(i=1,2,3,4)$ such that the following hold:
\beaa
&&\lim_{\xi\rightarrow+\infty}\frac{U(\xi)}{e^{\lambda^-_u(c^*)\xi}}=l_1,\\
&&\lim_{\xi\rightarrow+\infty}\frac{1-V(\xi)}{e^{\lambda_u^-(c^*)\xi}}=l_2\quad \mbox{if $\lambda_v^{-}(c^*)<\lambda_u^-(c^*)$},\\
&&\lim_{\xi\rightarrow+\infty}\frac{1-V(\xi)}{\xi e^{\lambda_v^{-}(c^*)\xi}}=l_3\quad
\mbox{if $\lambda_v^{-}(c^*)=\lambda_u^-(c^*)$},\\
&&\lim_{\xi\rightarrow+\infty}\frac{1-V(\xi)}{e^{\lambda_v^{-}(c^*)\xi}}=l_4\quad
\mbox{if $\lambda_v^{-}(c^*)>\lambda_u^-(c^*)$}.
\eeaa
\end{lemma}

To describe the asymptotic behavior of $(c^*,U,V)$ near $-\infty$, we define
\beaa
&&\mu_{u}^{-}(c):=\frac{-c-\sqrt{c^2+4}}{2}<0<\mu_{u}^{+}(c):=\frac{-c+\sqrt{c^2+4}}{2},\\
&&\mu_v^{-}(c):=\frac{-c-\sqrt{c^2+4rd(b-1)}}{2d}<0<\mu_v^{+}(c):=\frac{-c+\sqrt{c^2+4rd(b-1)}}{2d}.
\eeaa
\begin{lemma}[\cite{MoritaTachibana2009}]\label{lm: behavior around - infty}
Let $(c^*,U,V)$ be the minimal traveling wave of system (\ref{tw solution}). Then
there exist positive constants $l_i(i=5,6,7,8)$ such that
\beaa
&& \lim_{\xi\rightarrow-\infty}\frac{V(\xi)}{e^{\mu_v^+(c^*)\xi}}=l_5,\\
&&\lim_{\xi\rightarrow-\infty}\frac{1-U(\xi)}{e^{\mu_v^+(c^*)\xi}}=l_{6}\quad \mbox{if $\mu_{u}^{+}(c^*)>\mu_v^+(c^*)$},\\
&&\lim_{\xi\rightarrow-\infty}\frac{1-U(\xi)}{|\xi|e^{\mu_v^+(c^*)\xi}}=l_{7}\quad
\mbox{if $\mu_{u}^{+}(c^*)=\mu_v^+(c^*)$},\\
&&\lim_{\xi\rightarrow-\infty}\frac{1-U(\xi)}{e^{\mu_u^+(c^*)\xi}}=l_{8}\quad \mbox{if $\mu_{u}^{+}(c^*)<\mu_v^+(c^*)$}.
\eeaa
\end{lemma}

\subsubsection{Some useful estimates}

In this subsection, we provide 
some estimates for later use.
Note that the assumption {\bf(H2)} is not required in this subsection.

\begin{lemma}\label{lm: simple upper estimates}
Let $(u,v)$ be the solution of system \eqref{system} with initial data satisfying either \eqref{initial data} or \eqref{initial data2}. Then
there exist $M>0$ such that
\begin{equation*}
u(t,x)\le 1+Me^{-t}\ \ \mbox{and}\ \ v(t,x)\le 1+Me^{-rt}\quad\mbox{for all}\quad(t,x)\in[0,\infty)\times\mathbb{R}.
\end{equation*}
\end{lemma}
\begin{proof}
Since this can be done by simple comparison with ODEs, the proof is omitted.
\end{proof}

\begin{lemma}\label{lm: estimate of x=0}
Assume that {\bf(H1)} holds.
Let $(u,v)$ be the solution of system \eqref{system} with initial data satisfying either \eqref{initial data} or \eqref{initial data2}. Then
there exists $c_{uv}>0$ such that for any $c\in(0,c_{uv})$,
there exist $C_i>0$, $k_i>0$ $(i=1,2)$ and $T>0$ such that 
\begin{equation*}
u(t,x)\ge 1-C_1e^{-k_1t}\ \ \mbox{and}\ \  v(t,x)\leq C_2e^{-k_2t}\quad\mbox{for all}\quad(t,x)\in[T,\infty)\times[-ct,ct].
\end{equation*}
\end{lemma}
\begin{proof}
To derive these estimates, we consider the strong-strong (bistable) competition system
\begin{equation}\label{bistable competition system}
\left\{
\begin{aligned}
&\partial_t u^*=u^*_{xx}+u^*(1-u^*-a^*v^*),\ t\geq 0,\ x\in \mathbb{R},\\
&\partial_t v^*=dv^*_{xx}+rv^*(1-v^*-b^*u^*),\ t\geq 0,\ x\in \mathbb{R},\\
&u^*(0,x)=u(x,T_0),\quad v^*(0,x)=v(x,T_0),
\end{aligned}
\right.
\end{equation}
where $a^*>1>a$, $b^*=b>1$ and $T_0>0$ will be determined later.
Since $a^*>a$ and $b^*=b$, we have $N_1[u^*,v^*]\le 0$ and $N_2[u^*,v^*]\ge 0$ for $t\geq 0$. By applying the comparison principle, we have
\bea\label{CP uv and u*v*}
u^*\le u\quad {\rm and}\quad v^*\ge v\quad  \mbox{for all}\quad (t,x)\in [T_0,\infty)\times\mathbb{R}.
\eea

Now we fix the parameters $b^*=b$, $d$ and $r$, and let $a^*>1$ sufficiently close to $1$ such that
\bea\label{Ma-Huang-Ou-cond}
\frac{r+d(a^*-1)}{b^*r}<3-2a^*.
\eea
Due to \cite{Gardner,Kan-On}, there exists a unique $c_{uv}$ such that
system \eqref{tw solution} admits a unique traveling wave solution with $c=c_{uv}$, and $a$ is replaced $a^*$.
Furthermore, in view of \eqref{Ma-Huang-Ou-cond}, Theorem 4.3 in \cite{Ma Huang Ou} implies that $c_{uv}>0$.

On the other hand, by the results of \cite{Lewis Li Weinberger 1}, we can take $T_0>0$ large enough, such that
$(u,v)(T_0,\cdot)$ is close to $(1,0)$ in a sufficiently large interval, and thus the solution $(u^*, v^*)$ of the bistable
system \eqref{bistable competition system} satisfies that
$(u^*, v^*)(t, x)\to(1, 0)$ as $t\to\infty$ locally uniformly for  $x\in\mathbb{R}$ (see Remark 1.1 in \cite{Peng Wu Zhou}).
Together with $c_{uv}>0$, Lemma 2.6 and Lemma 2.8 in \cite{Peng Wu Zhou} is available to assert that for any $c\in[0,c_{uv})$,
\begin{equation}\label{u*v*-behavior}
u^*(t,x)\ge 1-C_1e^{-k_1t}\ \ \mbox{and}\ \  v^*(t,x)\leq C_2e^{-k_2t}\ \ \mbox{for all}\ \ (t,x)\in[T,\infty)\times[-ct,ct].
\end{equation}
for some $T>T_0$, $C_i>0$, $k_i>0$ ($i=1,2$).

Combining \eqref{CP uv and u*v*} and \eqref{u*v*-behavior}, we immediately obtain the desired result.
\end{proof}

Recall $\mathscr{C}$ from \eqref{definition of bar c}.

\begin{lemma}\label{lm: decay estimamte u and v}
Assume that {\bf(H1)} holds.
Let $(u,v)$ be the solution of system \eqref{system} with initial data \eqref{initial data2}.
Further, assume that  $c_v>c_u$. Then the following hold:
\begin{itemize}
  \item[(i)] for any $c>\mathscr{C}$, there exist $C_1,\nu_1,T_1>0$ such that
  \beaa
  \sup_{x\geq c t} u(t,x)\leq C_1e^{-\nu_1 t}\quad \mbox{for all}\quad t\geq T_1.
  \eeaa
  \item[(ii)] for any $c_1$ and $c_2$ with $\mathscr{C}<c_1<c_2<c_v$, there exist $C_2,\nu_2,T_2>0$ such that
  \beaa
  \sup_{c_1 t\leq x\leq c_2 t} v(t,x)\geq 1-C_2 e^{-\nu_2 t}\quad \mbox{for all}\quad t\geq T_2.
  \eeaa
\end{itemize}
\end{lemma}
\begin{proof}
 Let us briefly start with $(i)$. If $c>c_u$, then the conclusion is clear by comparing a super-solution of scalar KPP equation. Since $c_v>c_u$, it thus suffices to consider the case $\mathscr{C}<c<c_v$. In this case, the conclusion is  already included in  \cite[Proposition~1.5]{Girardin Lam} and the proof of \cite[Section 3.2.3, Theorem 1.1]{Girardin Lam}. We do not present the full details but only emphasize that a key tool is, for any small $\delta>0$, the {\it minimal} monotone traveling wave of the perturbed system
\begin{equation*}\label{perturbed tw solution}
\left\{
\begin{aligned}
&U''+cU'+U(1+\delta-U-aV)=0,\\
&dV''+cV'+rV(1-2\delta-V-bU)=0,\\
&(U,V)(-\infty)=(1+\delta,0),\ (U,V)(\infty)=(0,1-2\delta),\\
&U'<0,\ V'>0.
\end{aligned}
\right.
\end{equation*}

\medskip

Let us now turn to $(ii)$ for which the above perturbation argument seems unapplicable.  Let $\mathscr{C}<c_1<c_2<c_v$ be given. We only deal with $x\geq 0$. From \cite[Theorem 1.1]{Girardin Lam} we know
\begin{equation*}
    \lim_{t\to\infty}\sup_{c_1 t\leq  x\leq c_2t}\Big(u(t,x)+|1-v(t,x)|\Big)=0.
    \end{equation*}
From this and $(i)$, we can choose $\varepsilon>0$ small enough and $T_0\gg1$ such that
$$0<u(t,x)\leq C_1 e^{-\nu_1 t},\quad v(t,x)> 1-\varepsilon\quad\mbox{for all }\quad (t,x)\in[T_0,\infty)\times[c_1 t, c_2 t].$$
From the $v$-equation in system \eqref{system}, we have
\begin{equation}\label{v-ineq}
v_t\geq  d v_{xx}+r (1-\varepsilon)(1-v)-rbC_1v e^{-\nu_1 t}\quad\mbox{for all}\quad  (t,x)\in[T_0,\infty)\times[c_1 t, c_2 t].
\end{equation}
Defining
\begin{equation*}\label{def-tilde-v}
\tilde{v}(t,x):=v(t,x+\tilde{c}t),\quad \tilde{c}:=\frac{c_1+c_2}{2},
\end{equation*}
it follows  from \eqref{v-ineq} that
$$\tilde{v}_t\geq d \tilde{v}_{xx}+\tilde{c}\tilde{v}_x+r (1-\varepsilon)(1-\tilde{v})-rbC_1\tilde{v} e^{-\nu_1 t}\quad\mbox{for all}\quad  (t,x)\in[T_0,\infty)\times[-c_3 t, c_3 t],$$
where $c_3:=\frac{c_2-c_1}2$.

To estimate $\tilde{v}$, for any $T\geq T_0$, we define
$$
\alpha(t):=1+\frac{bC_1}{1-\varepsilon}e^{-\nu_1(t+T)}\quad\text{for all}\quad t\geq 0.
$$
Up to enlarging $T>0$ if necessary, we may assume $\alpha(0)<\frac 1{1-\varepsilon}$. Now, let us first consider the auxiliary problem
\begin{equation}\label{phi-sys}
\begin{cases}
\phi_t=d\phi_{xx}+\tilde{c}\phi_x+r (1-\varepsilon)[1-\alpha(t)\phi], &  t>0,\ -c_3T<x<c_3T,\\
\phi(t,\pm c_3{T})=1-\varepsilon, & t>0,\\
\phi(0,x)=1-\varepsilon, & -c_3{T}\leq x\leq c_3{T}.
\end{cases}
\end{equation}
Letting
$$\Phi(t,x):=e^{Q(t)}[\phi(t,x)-1+\varepsilon],\quad Q(t):=r(1-\varepsilon)t-\frac{rbC_1}{\nu_1}e^{-\nu_1(t+T)},$$
so that $Q'(t)=r(1-\varepsilon)\alpha(t)$, it follows from \eqref{phi-sys} that
\begin{equation}\label{Phi-sys}
\begin{cases}
\Phi_t=d\Phi_{xx}+\tilde{c}\Phi_x+r (1-\varepsilon)e^{Q(t)}(1-(1-\varepsilon)\alpha(t)), & t>0,\ -c_3T<x<c_3T,\\
\Phi(t,\pm c_3{T})=0, & t> 0,\\
\Phi(0,x)=0, & -c_3{T}\leq x\leq c_3{T}.
\end{cases}
\end{equation}
Up to a rescaling, we may assume $d=1$ so that \eqref{Phi-sys} is very comparable to \cite[problem (3.12)]{Kaneko Matsuzawa} on which we now rely. Denoting $G_1(t,x,z)$ the Green function of \cite[page 53]{Kaneko Matsuzawa} (with obvious changes of constants), we obtain the analogous of \cite[(3.14)]{Kaneko Matsuzawa}, namely
$$
\Phi(t,x)\ge r(1-\varepsilon)\int_0^te^{Q(s)}(1-(1-\varepsilon)\alpha(s))\left(\int_{-c_3T}^{c_3T}G_1(t-s,x,z)dz\right)ds,
$$
for all $t>0$, $ -c_3 T<x<c_3 T$. Next, for any small $0<\delta<1$, we define
$$
D_{\delta}:=\left\{(t,x)\in \mathbb{R}^2 : 0<t <  \delta^2c_3 {T},  \vert x\vert <(1-\delta)c_3 T\right\}.
$$
From the same process used in \cite[pages 54-55]{Kaneko Matsuzawa}, there exist $C_3, C_4>0$ such that the following lower estimate holds
$$
\Phi(t,x)\ge r(1-\varepsilon)(1-(1-\varepsilon)\alpha(0))(1-C_3e^{-C_4T})\int_0^te^{Q(s)}ds\quad\text{for all}\quad (t,x)\in D_\delta,
$$
resulting in
\begin{equation}\label{inequality of phi}
\phi(t,x)\ge \Psi(t) (1-C_3e^{-C_4T})(1-(1-\varepsilon)\alpha(0))+1-\varepsilon\quad\text{for all}\quad (t,x)\in D_\delta,
\end{equation}
where $\Psi(t):=r(1-\varepsilon)e^{-Q(t)}\int_0^te^{Q(s)}ds$. Denoting $K=\frac{rbC_1}{\nu_1}$, we have
\begin{equation*}
\begin{aligned}
\Psi(t)&\ge r(1-\varepsilon)e^{-Q(t)}\int_0^te^{r(1-\varepsilon)s}e^{-Ke^{-\nu_1T}}ds\\
&= e^{-r(1-\varepsilon)t}e^{Ke^{-\nu_1(t+T)}}e^{-Ke^{-\nu_1T}}\int_0^tr(1-\varepsilon)e^{r(1-\varepsilon)s}ds\\
&= e^{Ke^{-\nu_1T}(e^{-\nu_1t}-1)}(1-e^{-r(1-\varepsilon)t}).
\end{aligned}
\end{equation*}
Inserting this into \eqref{inequality of phi} and using  $e^y\ge 1+y$ for all $y\in\mathbb{R}$, we have, for all $(t,x)\in D_\delta$,
\begin{equation*}
\begin{aligned}
\phi(t,x)&\ge e^{Ke^{-\nu_1T}(e^{-\nu_1t}-1)}(1-e^{-r(1-\varepsilon)t})(1-C_3e^{-C_4T})(1-(1-\varepsilon)\alpha(0))+1-\varepsilon\\
&\ge (1-Ke^{-\nu_1T}(1-e^{-\nu_1t}))(1-e^{-r(1-\varepsilon)t})(1-C_3e^{-C_4T})(1-(1-\varepsilon)\alpha(0))+1-\varepsilon\\
&\ge (1-Ke^{-\nu_1T})(1-e^{-r(1-\varepsilon)t})(1-C_3e^{-C_4T})(1-(1-\varepsilon)\alpha(0))+1-\varepsilon.
\end{aligned}
\end{equation*}
Letting
$$
I_1:=1-Ke^{-\nu_1T}, \quad I_2(t):=1-e^{-r(1-\varepsilon)t}, \quad I_3:=1-C_3e^{-C_4T},
$$
we get
$$
\phi(t,x)\ge I_1I_2I_3(t)+(1-\varepsilon)(1-I_1I_2I_3(t)\alpha(0))\quad\text{for all}\quad(t,x)\in D_\delta.
$$
Now observe that  $I_1I_2I_3 \leq I_1 I_2(\delta ^2 c_3 T) I_3$. Furthermore some straightforward computations show that, if
\begin{equation}
\label{delta}
r(1-\varepsilon)\delta ^2c_3<\nu _1,
\end{equation}
then $I_1 I_2(\delta ^2 c_3 T) I_3\alpha(0)\leq 1$ up to enlarging $T>0$ if necessary. As a result, for all $(t,x) \in D_\delta$,
$$
\varphi(t,x) \geq I_1I_2I_3(t)\geq  1-K_1e^{-\nu_1T}-K_2e^{-r(1-\varepsilon)t},
$$
with some $K_1, K_2>0$. The last inequality holds since we can always choose $\nu_1<C_4$. As a conclusion, we have
\begin{equation}\label{phi-estiamte}
\phi(t,x)\geq 1-K_1e^{-\nu_1 T}-K_2e^{-r(1-\varepsilon)t}\quad\text{for all}\quad (t,x)\in D_{\delta},
\end{equation}
provided that $\delta>0$ is sufficiently small for \eqref{delta} to hold and  $T>0$ is sufficiently large.

In particular \eqref{phi-estiamte} implies that, for all $\vert x\vert \leq (1-\delta)c_3T$,
\begin{equation}\label{phi-estimate2}
\phi\left(\delta^2 c_3 {T},x\right) \geq  1-K_1e^{-\nu_1 {T}}- K_2 e^{-r(1-\varepsilon) \delta^2 c_3{T}}
                               \geq 1-(K_1+K_2) e^{-r(1-\varepsilon) \delta^2 c_3 {T}},
                                                                 \end{equation}
in virtue of \eqref{delta}.  On the other hand, we know from the comparison principle that $\tilde{v}(t+T,x)\geq \phi(t,x)$ for $t\geq0$ and $|x|\leq c_3 T$, which
together with \eqref{phi-estimate2} implies that
$$\tilde{v}\left(\delta^2c_3 T+T,x\right)\geq1-(K_1+K_2) e^{-r(1-\varepsilon) \delta^2 c_3 {T}}\quad \mbox{for all}\quad |x|\leq (1-\delta)c_3 T.$$
We further take
$t=(\delta^2c_3+1)T$,
which yields
$$\tilde{v}(t,x)\geq 1- C_2e^{-\nu_2t}\quad \mbox{for all large $t$ and $|x|\leq \frac{(1-\delta)c_3}{1+c_3\delta^2}t$},$$
where $C_2:=K_1+K_2$ and $\nu_2:=\frac{r(1-\varepsilon)\delta^2c_3}{1+c_3\delta^2}>0$. Recalling that $\tilde{v}(t,x)=v(t,x+\tilde{c}t)$ with $\tilde{c}=\frac{c_1+c_2}{2}$, that $c_3=\frac{c_2-c_1}{2}$ and since $\delta>0$ can be chosen arbitrarily small, the above estimate completes the proof of $(ii)$.
\end{proof}

The rest of this paper is organized as follows.
In section \ref{Sec:Cauchy prob-ic-1}, we will first study the Cauchy problem of system \eqref{system} with initial data \eqref{initial data} and prove Theorem~\ref{th: profile}. In section \ref{Sec:Cauchy prob-ic-2},
we focus on the Cauchy problem of system \eqref{system} with initial data \eqref{initial data2} and prove Theorem~\ref{th2: profile}.

\section{Cauchy problem with Scenario \eqref{initial data}}\label{Sec:Cauchy prob-ic-1}
\setcounter{equation}{0}

In this section, we shall prove Theorem~\ref{th: profile}.
The proof relies on delicate constructions of sub-solution and super-solution,
which are presented in subsection \ref{section: Construction of sub solution no ac}
and subsection~\ref{section: Construction of super solution no ac}, respectively.
The proof of Theorem~\ref{th: profile} is given in subsection~\ref{subsection:thm1}.

\subsection{Construction of sub-solution}\label{section: Construction of sub solution no ac}
We look for a sub-solution $(\underline{u},\bar{v})$ in the form of:
\begin{equation}\label{def of  sub sol no ac}
\left\{
\begin{aligned}
\underline{u}(t,x)&:=U(x-c^*t+\zeta(t))-P(t)\min\{e^{-\alpha_0(x-c^*t+x_0)},1\},\\
\bar{v}(t,x)&:=V(x-c^*t+\zeta(t))+Q(t),
\end{aligned}
\right.
\end{equation}
where $P(t)=p_0e^{-\mu_0 t}$, $Q(t)=q_0e^{-\mu_0 t}$, $\zeta(t)=\zeta_0-e^{-\tau_0 t}$, $(c^*,U,V)$ is the minimal traveling wave defined as \eqref{tw solution}. Here, all of the parameters are positive and will be determined in the following proof. For the simplicity, we denote $\xi:= x-c^*t+\zeta(t)$ and $W(t,x):=e^{-\alpha_0(x-c^*t+x_0)}$.

Clearly,
there exists a curve $\Gamma:[0,\infty)\to\mathbb{R}$ denoted by
$\Gamma(t):=c^*t-x_0$ such that $W(t,\Gamma(t))=1$ for all $t\ge0$.
Thus, it holds
$$W(t,x)\le 1\ \ \mbox{for}\ \ (t,x)\in S_1:=\{(t,x)\ |\ x\ge \Gamma(t)\},$$
$$W(t,x)\ge 1\ \ \mbox{for}\ \ (t,x)\in S_2:=\{(t,x)\ |\ x\le \Gamma(t)\}.$$
Then by some straightforward computation, we obtain that, for $(t,x)\in S_1$, it holds
\begin{equation}\label{N1 sub sol large x no ac}
\begin{aligned}
N_1[\underline{u},\bar{v}]:=&\partial_t\underline{u}-\underline{u}_{xx}-F(\underline{u},\bar{v})\\
=&\zeta'U'+(\mu_0-c^*\alpha_0+\alpha_0^2+1-2U-a(V+Q)+PW)PW+aQU.
\end{aligned}
\end{equation}
And for $(t,x)\in S_2$, it holds
\begin{equation*}\label{N1 sub sol small x no ac}
N_1[\underline{u},\bar{v}]=\zeta'U'+(\mu_0+1-2U-a(V+Q)+P)P+aQU.
\end{equation*}
On the other hand, for $(t,x)\in (0,\infty)\times\mathbb{R}$ it holds
\begin{equation}\label{N2 sub sol no ac}
\begin{aligned}
N_2[\underline{u},\bar{v}]:=&\partial_t\bar{v}-d\bar{v}_{xx}-G(\underline{u},\bar{v})\\
=&\zeta'V'+r(2V+Q+b(U-P\min\{W,1\})-1-\frac{\mu_0}{r})Q-brVP\min\{W,1\}.
\end{aligned}
\end{equation}

Next, we show that $(\underline{u},\bar{v})$ is a 
sub-solution by choosing suitable parameters.
In the following discussion, we choose $M>0$ sufficiently large and divide the whole space into three parts:
\begin{itemize}
	\item[(1)] $\Omega_1=\{M\le x-c^*t+\zeta(t)\}$;
	\item[(2)] $\Omega_2:=\{x-c^*t+\zeta(t)\le -M\}$;
	\item[(3)] $\Omega_3:=\{-M\le x-c^*t+\zeta(t)\le M\}$.
\end{itemize}

\noindent{\bf{Case 1}}: We first consider $(t,x)\in\Omega_1$.
Then, for some small $\delta>0$,
\beaa
\Omega_1=\{(t,x)\ |\ 0\le U\le \delta,\ 1-\delta\le V\le 1\}.
\eeaa
Note that, by setting
any $\zeta_0>0$ and $x_0=\zeta_0+M$, we have $\{x=\Gamma(t)\}\subset\Omega_2$ for all $t\ge0$, and hence $\Omega_1\subset S_1$.

Since $\delta>0$ can be chosen arbitrarily small and  $PW\to 0$ uniformly as $t\to\infty$, by setting
\begin{equation}\label{sub sol inequality condition 1 no ac}
\alpha_0\in(-\lambda^+_u(c^*),-\lambda_u^-(c^*)),
\end{equation}
we have
$(\mu_0-\alpha_0 c^*+\alpha_0^2+(1-a+a\delta)+PW)PW\le -C_0PW$
with some $\mu_0>0$ and $C_0>0$ for all large $t$. Moreover, by applying Lemma \ref{lm: behavior around + infty}, there exists $C_1>0$ such that
\begin{equation}\label{estimate for dU+}
U'\le -C_1U\ \ \mbox{for all}\ \ (t,x)\in \Omega_1.
\end{equation}
From \eqref{N1 sub sol large x no ac} and \eqref{estimate for dU+}, we can obtain that
\begin{equation*}
N_1[\underline{u},\bar{v}]\le-C_1\tau_0 e^{-\tau_0 t}U-C_0PW+aQU.
\end{equation*}
Therefore, by setting
\begin{equation}\label{sub sol inequality condition 2 no ac}
\tau_0<\mu_0,
\end{equation}
it holds $N_1[\underline{u},\bar{v}]\le 0$ for all $(t,x)\in \Omega_1$ and $t\ge T$ for some $T\gg 1$.

Next, we deal with the inequality of $N_2[\underline{u},\bar{v}]$. Since $V'>0$, $\zeta'>0$
and $\Omega_1\subset S_1$, \eqref{N2 sub sol no ac} implies that
\begin{equation*}\label{computation sub sol N2 no ac}
N_2[\underline{u},\bar{v}]\ge r(2V-bP-1-\frac{\mu_0}{r})Q-brPVW.
\end{equation*}
Since 
$V\ge 1-\delta$ in $\Omega_1$, by choosing $\mu_0<r/2$, it holds that
\beaa
N_2[\underline{u},\bar{v}]&\ge& r\Big[1-2\delta-bp_0e^{-\mu_0 t}-\frac{\mu_0}{r}\Big]q_0e^{-\mu_0 t}-brp_0e^{-\mu_0 t}W\\
&\ge& r \Big[\frac{1}{2}(1-\frac{2\mu_0}{r})q_0-bp_0W\Big]e^{-\mu_0 t}
\eeaa
for all large $t$, where $\delta>0$ is chosen smaller if necessary.
Therefore, by setting
\begin{equation}\label{sub sol inequality condition 3 no ac}
\mu_0<\frac{r}{2}\ \ \mbox{and}\ \ \frac{1}{2}(1-\frac{2\mu_0}{r})q_0> bp_0e^{-\alpha_0(M-\zeta_0+x_0)}=bp_0e^{-2\alpha_0 M}
\ \mbox{(since $x_0=\zeta_0+M$),}
\end{equation}
there exists $T\gg1$ such that $N_2[\underline{u},\bar{v}]\ge 0$ for all $(t,x)\in \Omega_1$ and $t\ge T$.

\bigskip
\noindent{\bf{Case 2}}: We consider $(t,x)\in\Omega_2$.
Then, for some small $\delta>0$,
\beaa
\Omega_2=\{(t,x)\ |\ 1\ge U\ge 1-\delta,\ \delta\ge V\ge 0\}.
\eeaa
Since $\zeta'>0$, $U'<0$ and $P\to 0$ uniformly as $t\to\infty$,  for $(t,x)\in \Omega_2\cap S_1$, from \eqref{N1 sub sol large x no ac}, we have
\beaa
N_1[\underline{u},\bar{v}]&\leq&(\mu_0-\alpha_0 c^*+\alpha_0^2+1-2U+PW)PW+aQU\\
&\leq& (\mu_0-\alpha_0 c^*+\alpha_0^2+PW)PW+[1-2(1-\delta)]PW+aQ.
\eeaa
Moreover, for $(t,x)\in \Omega_2\cap S_1$, we have 
$PW\ge Pe^{-\alpha_0(-M-\zeta(t)+x_0)}=Pe^{-\alpha_0(e^{-\tau_0 t})}$.
Therefore, by setting  $\alpha_0$ as \eqref{sub sol inequality condition 1 no ac} and
\begin{equation}\label{sub sol inequality condition 4 no ac}
\mu_0<1-a\ \ \mbox{and}\ 
p_0>aq_0,
\end{equation}
it holds $N_1[\underline{u},\bar{v}]\le 0$ for all $(t,x)\in S_1\cap\Omega_2$ and $t\ge T$ for some $T\gg1$.

On the other hand, for $(t,x)\in\Omega_2\cap S_2$, we have $\min\{W,1\}=1$ and thus
\begin{equation*}
N_1[\underline{u},\bar{v}]\le (\mu_0+1+P-2U)P+aQU\le [2-a+P-2(1-\delta)]P+aQ,
\end{equation*}
where we used $\mu_0<1-a$.
Therefore, by setting
\begin{equation}\label{sub sol inequality condition 5 no ac}
\mu_0<1\ \ \mbox{and}\ \ 
p_0>q_0,
\end{equation}
there exists $T\gg1$ such that
$N_1[\underline{u},\bar{v}]\le 0$ for all $(t,x)\in S_2\cap\Omega_2$ and $t\ge T$ for some $T\gg 1$.

Next, we will deal with the inequality of $N_2[\underline{u},\bar{v}]$. To verify $N_2[\underline{u},\bar{v}]\ge 0$, we first observe that, from Lemma \ref{lm: behavior around - infty} there exists $C_4>0$ such that $\zeta'V'\ge C_4\zeta'V$.
Then, from \eqref{N2 sub sol no ac}, we have
$$N_2[\underline{u},\bar{v}]\ge C_4\zeta'V+r(bU-1-bP-\frac{\mu_0}{r})Q-brPV.$$
Thus, by setting $\mu_0<r(b-1)$, since $P\to 0$ uniformly as $t\to\infty$, we have
$$N_2[\underline{u},\bar{v}]\ge C_4\zeta'V-brPV.$$
Therefore, by setting
\begin{equation}\label{sub sol inequality condition 6 no ac}
\tau_0<\mu_0<r(b-1),
\end{equation}
it holds $N_2[\underline{u},\bar{v}]\ge 0$ for all $(t,x)\in \Omega_2$ and $t\ge T$ for some $T\gg 1$.

\bigskip
\noindent{\bf{Case 3}}: We consider $(t,x)\in\Omega_3$.
Then, for some small $\delta_i>0$ ($i=1,2$),
\beaa
\Omega_3=\{(t,x)\ |\ 1-\delta_1\ge U\ge \delta_2,\ 1-\delta_2\ge V\ge \delta_1\}.
\eeaa
In this range, there exists $C_2>0$ such that $U'\le-C_2$, which implies that $\zeta'U'\le-C_2\zeta'$.  Therefore, we have
\begin{equation*}
\begin{aligned}
N_1[\underline{u},\bar{v}]\le-C_2\zeta'+C_3P+aQU.
\end{aligned}
\end{equation*}
Then, for $\tau_0$ and $\mu_0$ satisfying \eqref{sub sol inequality condition 2 no ac},
it holds $N_1[\underline{u},\bar{v}]\le 0$ for all $(t,x)\in \Omega_3$ and $t\ge T$ for some $T\gg1$.

Next, we deal with the inequality of $N_2[\underline{u},\bar{v}]$. Note that, in this range, we have $V'\ge C_4>0$, which implies that $\zeta'V'\ge C_5\zeta'$. Therefore, we have
$$N_2[\underline{u},\bar{v}]\ge C_4\zeta'-C_5Q-brP.$$
Similarly, for $\tau_0$ and $\mu_0$ satisfying \eqref{sub sol inequality condition 2 no ac},
it holds $N_2[\underline{u},\bar{v}]\ge 0$ for all $(t,x)\in \Omega_3$ and $t\ge T$ for some $T\gg1$.

By concluding the conditions \eqref{sub sol inequality condition 1 no ac}, \eqref{sub sol inequality condition 2 no ac}, \eqref{sub sol inequality condition 3 no ac}, \eqref{sub sol inequality condition 4 no ac}, \eqref{sub sol inequality condition 5 no ac}, \eqref{sub sol inequality condition 6 no ac} provided from the discussion above, we get a key lemma as follow:
\begin{lemma}\label{lm: sub sol inequality}
For any $\alpha_0, \mu_0, \tau_0, x_0, \zeta_0>0$ satisfying
\begin{itemize}
	\item[(1)] $\alpha_0\in(-\lambda_u^+(c^*),-\lambda_u^-(c^*))$,
	\item[(2)] $\tau_0<\mu_0<\min\{1-a,r(b-1),\frac{r}{2}\}$,
	\item[(3)] $x_0-\zeta_0>0$ sufficiently large,
\end{itemize}
then there exists $p_0>0$, $q_0>0$ and $T\ge 0$ such that
$$N_1[\underline{u},\bar{v}]\le 0 \ \ \mbox{and}\ \ N_2[\underline{u},\bar{v}]\ge 0\ \ \mbox{in}\ \ [T,\infty)\times\mathbb{R},$$
where $(\underline{u},\bar{v})$ is defined as \eqref{def of  sub sol no ac} with $c=c^*$.
\end{lemma}

\begin{remark}\label{rk: mu for subsol}
Note that, from the proof of Lemma~\ref{def of  sub sol no ac}, $T$  always can be chosen independently of all small $\mu_0>0$.
More precisely, if $\mu_0>0$ becomes smaller, we can choose smaller $p_0$ and $q_0$ such that
the differential inequalities still holds for $t\ge T$ with the same $T$.
\end{remark}

\begin{lemma}\label{lm: sub sol boundary condition no ac}
Let $(\underline{u},\bar{v})$ be defined as \eqref{def of  sub sol no ac}
with $\alpha_0$, $\mu_0$, $\tau_0$, $\zeta_0$, $x_0$, $p_0$, $q_0$ satisfying the conditions in Lemma \ref{lm: sub sol inequality}. Then there exist $T_0,T^{*}>0$ such that the solution $(u,v)$ with initial data \eqref{initial data}
satisfies
$$u(t+T^*,x)\ge\underline{u}(t,x),\ v(t+T^*,x)\le\bar{v}(t,x)\quad\mbox{for all}\quad(t,x)\in[T_0,\infty)\times[0,\infty).$$
\end{lemma}
\begin{proof}
Let the parameters satisfy the conditions in Lemma \ref{lm: sub sol inequality}, then we have
$$N_1[\underline{u},\bar{v}]\le 0 \ \ \mbox{and}\ \ N_2[\underline{u},\bar{v}]\ge 0\ \ \mbox{in}\ \ [T_0,\infty)\times[0,\infty).$$

Let us fix this $T_0>0$.
From Lemma \ref{lm: estimate of x=0}, by setting $\mu_0<\min\{k_1,k_2\}$ (see Remark~\ref{rk: mu for subsol}),
there exists $T_1\gg1$ such that
\beaa
&&\underline{u}(t,0)\le 1-p_0e^{-\mu_0 t}\le 1-C_1e^{-k_1(t+T_1)}\le u(t+T_1,0)\quad\text{for all}\quad t\geq T_0,\\
&&\bar{v}(t,0)\ge q_0e^{-\mu_0 t}\ge e^{-k_2(t+T_1)}\ge v(t+T_1,0)\quad\text{for all}\quad t\geq T_0.
\eeaa
Next, by the definition of $(\underline{u},\bar{v})$, condition (1) of Lemma \ref{lm: sub sol inequality} and Lemma \ref{lm: behavior around + infty}, we can choose $\ell\gg1$ such that
\beaa
\underline{u}(T_0,x)=0,\quad \bar{v}(T_0,x)\geq 1+\frac{q_0}{2} e^{-\mu_0 T_0}\quad \mbox{for all}\quad x\geq \ell.
\eeaa
Then,
by Lemma~\ref{lm: simple upper estimates}, we can take $T_2\gg1$ such that $u(T_0+T_2,x)\geq \underline{u}(T_0,x)$ and
$v(T_0+T_2,x)\leq \bar{v}(T_0,x)$ for all $x\geq \ell$.

For $x\in[0,\ell]$, we can choose $T_3\gg1$ such that
\beaa
&&\underline{u}(T_0,x)\leq 1-P(T_0)\min\{e^{-\alpha_0(x-c^* T_0 +x_0)},1\}\leq u(t+T_3,x)\quad\text{for all}\quad x\in[0,\ell],\\
&&v(T_0+T_3,x) \leq \min_{x\in[0,\ell]}\bar{v}(T_0,x)\leq \bar{v}(T_0,x)\quad \text{for all}\quad x\in[0,\ell],
\eeaa
since $(u,v)\to (1,0)$ as $t\to\infty$ uniformly for $x\in[0,\ell]$.

By the above discussion and setting $T^*=\max\{T_1,T_2,T_3\}$,
we can assert that
\beaa
&&u(T_0+T^*,x)\geq\underline{u}(T_0,x),\ v(T_0+T^*,x)\leq \bar{v}(T_0,x)\quad\text{for all}\quad  x\in[0,\infty),\\
&&u(t+T^*,0)\geq\underline{u}(t,0),\ v(t+T^*,0)\leq \bar{v}(t,0)\quad \text{for all}\quad t\geq T_0.
\eeaa
Therefore, by applying the comparison principle, the proof is complete.
\end{proof}

Furthermore, if we consider a sub-solution $(\underline{u}_*,\bar{v}_*)$ defined as
\begin{equation}\label{def of sub sol no ac x<0}
\left\{
\begin{aligned}
\underline{u}_*(t,x)&:=U(-x-c^*t+\zeta(t))-P(t)\min\{W(t,-x),1\},\\
\bar{v}_*(t,x)&:=V(-x-c^*t+\zeta(t))+Q(t),
\end{aligned}
\right.
\end{equation}
then, by repeating the above argument, we can obtain a lemma as follow:
\begin{lemma}\label{lm: sub sol boundary condition x<0 no ac}
Let $(\underline{u}_*,\bar{v}_*)$ be defined as \eqref{def of sub sol no ac x<0} with $\alpha_0$, $\mu_0$, $\tau_0$, $\zeta_0$, $x_0$, $p_0$, $q_0$ satisfying the conditions in Lemma \ref{lm: sub sol inequality}. Then there exist $T_0,T^{*}>0$ such that the solution $(u,v)$ with initial data \eqref{initial data}
satisfies
$$u(t+T^*,x)\ge\underline{u}_*(t,x),\ v(t+T^*,x)\le\bar{v}_*(t,x)\ \ \mbox{for all}\ \ (t,x)\in[T_0,\infty)\times(-\infty,0].$$
\end{lemma}

\subsection{Construction of super-solution}\label{section: Construction of super solution no ac}
We look for a super-solution $(\bar{u},\underline{v})$  in the form of
\begin{equation}\label{def of super sol}
\left\{
\begin{aligned}
&\bar{u}(t,x):=U(x-c^*t-\zeta(t))+P(t)\min\{e^{-\alpha_1(x-c^*t+x_1)},1\},\\
&\underline{v}(t,x):=V(x-c^*t-\zeta(t))-Q(t),
\end{aligned}
\right.
\end{equation}
where $\zeta(t)=\zeta_1-e^{-\tau_1 t}$, $Q(t)=q_1e^{-\mu_1 t}$, $P(t)=p_1e^{-\mu_1 t}$.
The parameters $\alpha_1$, $\mu_1$, $\tau_1$, $p_1$, $q_1$, $\zeta_1$ are positive constants and will be determined later. For the simplicity, we denote $\xi:=x-c^*t-\zeta(t)$ and $W(t,x)=e^{-\alpha_1(x-c^*t+x_1)}$.  Next, we show that $(\bar{u},\underline{v})$ is a super-solution by choosing suitable parameters.


Clearly, there exists a curve $\Gamma(t)=c^*t-x_1:[0,\infty)\to\mathbb{R}$ such that $W(t,\Gamma(t))=1$.
Thus, it holds
$$W(t,x)\le 1\ \ \mbox{for}\ \ (t,x)\in S_1:=\{(t,x)\ |\ x\ge \Gamma(t)\},$$
$$W(t,x)\ge 1\ \ \mbox{for}\ \ (t,x)\in S_2:=\{(t,x)\ |\ x\le \Gamma(t)\}.$$
Then by some straightforward computation, we obtain that, for $(t,x)\in S_1$, it holds
\begin{equation}\label{N1 super sol inequality S1}
N_1[\bar{u},\underline{v}]=\big(-1-\mu_1+\alpha_1 c^*-\alpha_1^2+2U+a\underline{v}+PW\big)PW-\zeta'(t)U'-aQU.
\end{equation}
And for $(t,x)\in S_2$, it holds
\begin{equation*}\label{N1 super sol inequality S2}
N_1[\bar{u},\underline{v}]=\big(-1-\mu_1+2U+a\underline{v}+P\big)P-\zeta'(t)U'-aQU.
\end{equation*}
On the other hand, for $(t,x)\in [0,\infty)\times\mathbb{R}$ it holds
\begin{equation}\label{N2 super sol inequality}
N_2[\underline{u},\bar{v}]=-\zeta'V'+r(1-2V+Q-b\bar{u}+\frac{\mu_1}{r})Q+brPV\min\{W,1\}.
\end{equation}

In the following discussion, we choose $M>0$ sufficiently large and divide the whole space into three parts:
\begin{itemize}
	\item[(1)] $\Omega_1=\{ x\ge c^*t-\zeta(t)+ M\}$;
	\item[(2)] $\Omega_2:=\{c^*t-\zeta(t)-M\ge x\ge 0\}$;
	\item[(3)] $\Omega_3:=\{c^*t-\zeta(t)+M\ge x\ge ct-\zeta(t)-M\}$.
\end{itemize}

\noindent{\bf{Case 1}}: We consider $(t,x)\in\Omega_1$ with $M>0$ sufficiently large.
Then, for some small $\delta>0$,
\beaa
\Omega_1=\{(t,x)\ |\ 0\le U\le \delta,\ 1-\delta\le V\le 1\}.
\eeaa
Note that, by setting $x_1=\zeta_1+M$, then for sufficiently large $t$, we have $\Gamma(t)\subset\Omega_2$, and hence $\Omega_1\subset S_1$.

Since $\delta$ can be chosen arbitrarily small, by setting $\alpha_1$ as \eqref{sub sol inequality condition 1 no ac},
we have
$$-(\mu_1-c\alpha_1+\alpha_1^2+1-2U-a(V-Q)-PW)PW\ge C_0PW$$
with some $\mu_1>0$ and $C_0>0$. Moreover, by applying Lemma \ref{lm: behavior around + infty} and Lemma \ref{lm: behavior around - infty}, there exists $C_1>0$ such that \eqref{estimate for dU+} holds.
From \eqref{estimate for dU+} and  \eqref{N1 super sol inequality S1}, we can obtain that
\begin{equation*}
N_1[\bar{u},\underline{v}]\ge C_1\zeta_1\tau_1 e^{-\tau_1 t}U-C_0PW-aQU.
\end{equation*}
Therefore, by setting
\begin{equation}\label{super sol inequality condition 2}
\tau_1<\mu_1,
\end{equation}
it holds $N_1[\bar{u},\underline{v}]\ge 0$ for $(t,x)\in \Omega_1$ and $t\ge T$ for some $T\gg 1$.

Next, we deal with the inequality of $N_2[\bar{u},\underline{v}]$. Since $V'>0$ and $\zeta'>0$, from \eqref{N2 super sol inequality}, we have
\begin{equation*}\label{computation super sol N2}
N_2[\bar{u},\underline{v}]\le r(1-2V+Q+\frac{\mu_1}{r})Q+brVP\min\{W,1\}.
\end{equation*}
Since $Q\to 0$ uniformly  as $t\to\infty$ and $V\ge 1-\delta$ in $\Omega_1$,  by choosing $\mu_1<\frac{r}{2}$, it holds
$$N_2[\bar{u},\underline{v}]\le -\frac{r}{4}Q+brPW.$$
Then, by setting
\begin{equation}\label{super sol inequality condition 3}
\mu_1<\frac{r}{2}\ \ \mbox{and}\ \ \frac{q_1}{4}> bp_1e^{-\alpha_1(M-\zeta_1+x_1)}=bp_1e^{-2\alpha_1 M},
\end{equation}
it holds $N_2[\bar{u},\underline{v}]\le 0$ for $(t,x)\in \Omega_1$ and $t\ge T$ for some $T\gg 1$.


\bigskip
\noindent{\bf{Case 2}}: We consider $(t,x)\in\Omega_2$.
Then, for some small $\delta>0$, $\Omega_2=\{(t,x)\ |\ 1\ge U\ge 1-\delta,\ \delta\ge V\ge 0\}$.
Since $\zeta'>0$, $U'<0$ and $P\to 0$ uniformly as $t\to\infty$,  for $(t,x)\in \Omega_2\cap S_1$, from \eqref{N1 super sol inequality S1}, we have
\begin{equation*}
N_1[\bar{u},\underline{v}]\ge-(\mu_1-c\alpha_1+\alpha_1^2)PW-aQU.
\end{equation*}
Moreover, for $(t,x)\in \Omega_2\cap S_1$, we have $PW\ge p(t)e^{-\alpha_1(x_1-\zeta_1-M)}$. Therefore, by setting  $\alpha_1$ as \eqref{sub sol inequality condition 1 no ac} and
\begin{equation}\label{super sol inequality condition 4}
\mu_1<1-a\ \ \mbox{and}\ \ aq_1< C_2p_1e^{-\alpha_1(x_1-\zeta_1-M)}=C_2p_1,
\end{equation}
it holds $N_1[\bar{u},\underline{v}]\ge 0$ for $(t,x)\in S_1\cap\Omega_2$ and $t\ge T$ for some $T\gg 1$.

On the other hand, for $(t,x)\in\Omega_2\cap S_2$, we have
\begin{equation*}
N_1[\bar{u},\underline{v}]\ge-(\mu_1+1+P-2U)P-aQU\ge C_3P-aQU.
\end{equation*}
Therefore, by setting
\begin{equation}\label{super sol inequality condition 5}
\mu_1<1\ \ \mbox{and}\ \ aq_1< C_3p_1,
\end{equation}
it holds $N_1[\bar{u},\underline{v}]\ge 0$ for $(t,x)\in S_2\cap\Omega_2$ and $t\ge T$ for some $T\gg 1$.

Next, we will deal with the inequality of $N_2[\bar{u},\underline{v}]$. To verify $N_2[\bar{u},\underline{v}]\le 0$, we first observe that, there exists $C_4>0$ such that $-\zeta'V'\le -C_4\zeta'V$.
Then, from \eqref{N2 super sol inequality}, we have
$$N_2[\bar{u},\underline{v}]\le -C_4\zeta'V+r(1-bU+Q+\frac{\mu}{r})Q+brPV.$$
Thus, by choosing $\mu_1<r(b-1)$, since $P\to 0$ uniformly as $t\to\infty$, we have
$$N_2[\bar{u},\underline{v}]\le -C_4\zeta'V+brPV.$$
Therefore, by setting
\begin{equation}\label{super sol inequality condition 6}
\tau_1<\mu_1<r(b-1),
\end{equation}
it holds $N_2[\bar{u},\underline{v}]\le 0$ for $(t,x)\in \Omega_2$ and $t\ge T$ for some $T\gg 1$.

\bigskip
\noindent{\bf{Case 3}}: We consider $(t,x)\in\Omega_3$.
Then, for some small $\delta_1,\delta_2>0$, $\Omega_3=\{(t,x)\ |\ 1-\delta_1\ge U\ge \delta_1,\ 1-\delta_2\ge V\ge \delta_2\}$. In this range, there exists $C_5>0$ such that $U'\le-C_5$, which implies that $-\zeta'U'\ge C_5\zeta'$.  Therefore, we have
\begin{equation*}
N_1[\bar{u},\underline{v}]\ge C_5\zeta'-C_6P-aQU.
\end{equation*}
Then, for $\tau_1$ and $\mu_1$ satisfying \eqref{super sol inequality condition 2},
it holds $N_1[\bar{u},\underline{v}]\ge 0$ for $(t,x)\in \Omega_3$ and $t\ge T$ for some $T\gg 1$.

Next, we will deal with the inequality of $N_2[\bar{u},\underline{v}]$. We observe that, in this range, we have $V'\ge C_7>0$, which implies that $-\zeta'V'\le -C_7\zeta'$. Therefore, we have
$$N_2[\bar{u},\underline{v}]\le -C_7\zeta'+C_8Q+brP.$$
Similarly, for $\tau_1$ and $\mu_1$ satisfying \eqref{super sol inequality condition 2},
it holds $N_2[\bar{u},\underline{v}]\le 0$ for $(t,x)\in \Omega_3$ and $t\ge T$ for some $T\gg1$.

By concluding the conditions \eqref{sub sol inequality condition 1 no ac}, \eqref{super sol inequality condition 2}, \eqref{super sol inequality condition 3}, \eqref{super sol inequality condition 4}, \eqref{super sol inequality condition 5}, \eqref{super sol inequality condition 6} provided from the discussion above, we get a key lemma as follow:
\begin{lemma}\label{lm: super sol inequality}
For any $\alpha_1, \mu_1, \tau_1, x_1, \zeta_1>0$ satisfying
\begin{itemize}
	\item[(1)] $\alpha_1\in(-\lambda_u^+(c^*),-\lambda_u^-(c^*))$,
	\item[(2)] $\tau_1<\mu_1<\min\{1-a,r(b-1),\frac{r}{2}\}$,
	\item[(3)] $x_1-\zeta_1$ sufficiently large,
\end{itemize}
then there exists $p_1>0$, $q_1>0$ and $T\ge 0$ such that
$$N_1[\bar{u},\underline{v}]\ge 0 \ \ \mbox{and}\ \ N_2[\bar{u},\underline{v}]\le 0\ \ \mbox{in}\ \ [T,\infty)\times\mathbb{R},$$
where $(\bar{u},\underline{v})$ is defined as \eqref{def of  super sol}.
\end{lemma}


\begin{lemma}\label{lm: super sol boundary condition}
Let $(\bar{u},\underline{v})$ be defined as \eqref{def of  super sol},
and $\alpha_1$, $\mu_1$, $\tau_1$, $\zeta_1$, $x_1$, $p_1$, $q_1$ satisfies the conditions in Lemma \ref{lm: super sol inequality}. Then there exist $T^{**}>0$ such that the solution $(u,v)$ with initial data \eqref{initial data}
satisfies
$$u(t,x)\le\bar{u}(t+T^{**},x),\ v(t,x)\ge\underline{v}(t+T^{**},x)\quad\mbox{for all}\quad (t,x)\in[0,\infty)\times[0,\infty).$$
\end{lemma}
\begin{proof}
First, by setting the parameters as that in Lemma \ref{lm: super sol inequality}, we have
$$N_1[\bar{u},\underline{v}]\ge 0 \ \ \mbox{and}\ \ N_2[\bar{u},\underline{v}]\le 0\ \ \mbox{in}\ \ [T^{**},\infty)\times[0,\infty).$$
Let us fix this $T^{**}>0$.

Then, from Lemma \ref{lm: behavior around - infty} and Lemma \ref{lm: simple upper estimates}, for any $\mu_1<\min\{1,c^*\mu_u^+(c^*),c^*\mu_v^+(c^*)\}$,
there exists $\zeta_1>0$ such that for $t\ge 0$, we have
$$\bar{u}(t+T^{**},0)=U(-c^*(t+T^{**})-\zeta(t+T^{**}))+P(t+T^{**})\ge 1\ge u(t,0),$$
$$\underline{v}(t+T^{**},0)=V(-c^*(t+T^{**})+\zeta(t+T^{**}))-Q(t+T^{**})\le 0\le v(t,0).$$

Next, we fix this $\mu_1$. Then, from the construction of $(\bar{u},\underline{v})$ and \eqref{initial data}, up to increasing $\zeta_1$ if it is necessary, we have
\begin{equation*}
u(0,x)\le\bar{u}(T^{**},x)\ \ \mbox{and}\ \ v(0,x)\ge\underline{v}(T^{**},x)\quad\mbox{for all}\quad x\in[0,\infty).
\end{equation*}
Therefore, by applying the comparison principle, the proof is complete.
\end{proof}

Furthermore, if we consider a sub-solution $(\bar{u}_*,\underline{v}_*)$ defined as
\begin{equation}\label{def of super sol no ac x<0}
\left\{
\begin{aligned}
\bar{u}_*(t,x)&:=U(-x-c^*t+\zeta(t))+P(t)\min\{W(t,-x),1\},\\
\underline{v}_*(t,x)&:=V(-x-c^*t+\zeta(t))-Q(t),
\end{aligned}
\right.
\end{equation}
then, by repeating the above argument, we can obtain a lemma as follow:
\begin{lemma}\label{lm: super sol boundary condition x<0 no ac}
Let $(\bar{u}_*,\underline{v}_*)$ be defined as \eqref{def of super sol no ac x<0}  and $\alpha_1$, $\mu_1$, $\tau_1$, $\zeta_1$, $x_1$, $p_1$, $q_1$ satisfy the conditions in Lemma \ref{lm: super sol inequality}. Then there exists $T^{**}>0$ such that the solution $(u,v)$ with initial data \eqref{initial data}
satisfies
$$u(t,x)\le\bar{u}_*(t+T^{**},x),\ v(t,x)\ge\underline{v}_*(t+T^{**},x)\ \ \mbox{for all}\ \ (t,x)\in[0,\infty)\times(-\infty,0].$$
\end{lemma}
\subsection{Proof of Theorem \ref{th: profile}}\label{subsection:thm1}
Let us set $\xi:=x-c^*t$. Then we can write the solution of system \eqref{system} as
\begin{equation}\label{def of tilde u v}
(\tilde{u},\tilde{v})(t,\xi)=(u,v)(t,x)=(u,v)(t,\xi+c^*t),\ \ t> 0,\ \xi\in\mathbb{R},
\end{equation}
which satisfies
\begin{equation*}\label{system moving frame}
\left\{
\begin{aligned}
&\partial_t\tilde{u}=\tilde{u}_{\xi\xi}+c^*\tilde{u}_{\xi}+\tilde{u}(1-\tilde{u}-a\tilde{v}),\ \ t>0,\ \xi\in\mathbb{R},\\
&\partial_t\tilde{v}=d\tilde{v}_{\xi\xi}+c^*\tilde{v}_{\xi}+r\tilde{v}(1-\tilde{v}-b\tilde{u}),\ \ t>0,\ \xi\in\mathbb{R}.
\end{aligned}
\right.
\end{equation*}
Thanks to Lemma \ref{lm: sub sol boundary condition no ac} and Lemma \ref{lm: super sol boundary condition}, we can immediately obtain the following result.
\begin{lemma}\label{lm: lower and upper esitimates of tilde u v}
Let $(c^*,U,V)$ be the minimal traveling wave of system \eqref{tw solution}. Then there exist constants $p_2$, $q_2$, $\alpha_2$, $\mu_2$, $x_2$, $\zeta_2$ and $K_i(i=1,2,3,4)$, and $T>0$ such that
\begin{equation*}\label{lower and upper esitimates of tilde u v}
\left\{
\begin{aligned}
&\tilde{u}(t,\xi)\ge U\left(\xi+\zeta_2-K_1e^{-\frac{\mu_2}{2}t}\right)-K_2p_2e^{-\mu_2 t}\min\{e^{-\alpha_2(\xi+x_2)},1\},\\
&\tilde{u}(t,\xi)\le U\left(\xi-\zeta_2+K_3e^{-\frac{\mu_2}{2}t}\right)+K_4p_2e^{-\mu_2 t}\min\{e^{-\alpha_2(\xi+x_2)},1\},\\
&\tilde{v}(t,\xi)\le V\left(\xi+\zeta_2-K_1e^{-\frac{\mu_2}{2}t}\right)+K_2q_2e^{-\mu_2 t},\\
&\tilde{v}(t,\xi)\ge V\left(\xi-\zeta_2+K_3e^{-\frac{\mu_2}{2}t}\right)-K_4q_2e^{-\mu_2 t}.
\end{aligned}
\right.
\end{equation*}
for $\xi\ge-c^* t$ and $t\ge T$.
\end{lemma}

From the construction of sub-solution and super-solution, we actually establish the local stability of traveling waves
in the following sense:

\begin{lemma}\label{lem:asymptotic stability of tw}
Let $(c^*,U,V)$ be a solution of \eqref{tw solution}. Then there exists a function $\nu(\varepsilon)$ defined for small $\varepsilon$ with $\nu(\varepsilon)\to 0$ as $\varepsilon\to 0$ satisfying the following property: if
\begin{equation*}
\left|\tilde{u}(t_*,\xi)-U(\xi-\xi_*)\right|+\left|\tilde{v}(t_*,\xi)-V(\xi-\xi_*)\right|<\varepsilon\quad\mbox{for all}\quad\xi\in\mathbb{R},
\end{equation*}
for some $t_*,\xi_*\in\mathbb{R}$, then
\begin{equation*}
\left|\tilde{u}(t,\xi)-U(\xi-\xi_*)\right|+\left|\tilde{v}(t,\xi)-V(\xi-\xi_*)\right|<\nu(\varepsilon)\quad\mbox{for all}\quad (t,\xi)\in[t_*,\infty)\times\mathbb{R}.
\end{equation*}
\end{lemma}
\begin{proof}
 In the proof of  Lemma~\ref{lm: sub sol inequality} and Lemma~\ref{lm: super sol inequality},
we may choose $q_i=O(\varepsilon)$, $p_i=O(\varepsilon)$ and $|\zeta_i-\xi_*|=O(\varepsilon)$, $i=0,1$,
such that $(u, v)(t, x)$
can be compared with the sub-solution and super-solution constructed in Lemma~\ref{lm: sub sol inequality} and Lemma~\ref{lm: super sol inequality} in terms of $(U,V)$, respectively, from $t=t_*$. Therefore, Lemma~\ref{lem:asymptotic stability of tw} follows from the comparison principle.
\end{proof}

Now we are ready to prove Theorem \ref{th: profile}. Let $(\tilde{u},\tilde{v})$ defined as \eqref{def of tilde u v} and $(c^*, U, V)$ be the minimal traveling wave of system \eqref{tw solution}. Let $\{t_n\}$ be an arbitrary sequence satisfying $t_n\to\infty$ as $n\to\infty$. Set
$$\tilde{u}_n(t,\xi)=\tilde{u}(t+t_n,\xi),\quad \tilde{v}_n(t,\xi)=\tilde{v}(t+t_n,\xi),\ n\in \mathbb{N}.$$
By the standard parabolic regularity theory, up to extraction of a subsequence, we have $(\tilde{u}_n,\tilde{v}_n)\to(u^{\infty},v^{\infty})$ locally uniformly as $n\to\infty$, and $(u^{\infty},v^{\infty})$ satisfies
\begin{equation}\label{limit system moving frame}
\left\{
\begin{aligned}
&\partial_tu^{\infty}=u^{\infty}_{\xi\xi}+c^*u^{\infty}_{\xi}+u^{\infty}(1-u^{\infty}-av^{\infty}),\ \ t\in\mathbb{R},\ \xi\in\mathbb{R},\\
&\partial_tv^{\infty}=dv^{\infty}_{\xi\xi}+c^*v^{\infty}_{\xi}+rv^{\infty}(1-v^{\infty}-bu^{\infty}),\ \ t\in\mathbb{R},\ \xi\in\mathbb{R}.
\end{aligned}
\right.
\end{equation}
In addition, by replacing $t$ by $t+t_n$ in the inequalities of Lemma \ref{lm: lower and upper esitimates of tilde u v}, we have, for all $(t,\xi)\in\mathbb{R}\times\mathbb{R}$,
\begin{equation}\label{u v infty betweent two tw}
U(\xi+\zeta_2)\le u^{\infty}(t,\xi)\le U(\xi-\zeta_2)\quad\text{and}\quad V(\xi-\zeta_2)\le v^{\infty}(t,\xi)\le V(\xi+\zeta_2).
\end{equation}

Note that \eqref{u v infty betweent two tw} indicates that
$(u^{\infty}, v^{\infty})$ is trapped between two shifts of the minimal traveling wave. The following lemma
shows that $(u^{\infty}, v^{\infty})$ is exactly the minimal wave with a translation. The proof is based on a sliding method (see \cite{Berestycki Hamel2007}).

\begin{lemma}\label{lm:u v infity converges to tw}
There exists $\tilde{\zeta}\in[-\zeta_2, \zeta_2]$ such that
\begin{equation*}\label{u v infity converges to tw}
u^{\infty}(t,\xi)=U(\xi-\tilde{\zeta})\quad\mbox{and}\quad v^{\infty}(t,\xi)=V(\xi-\tilde{\zeta})\quad\mbox{for all}\quad (t,\xi)\in\mathbb{R}\times\mathbb{R}.
\end{equation*}
\end{lemma}
\begin{proof}
We choose $\delta>0$ small and let $A>0$ such that
\bea\label{U-V-right plane}
1-\delta\le U(\xi+\zeta_2)\le 1,\quad 0\le V(\xi+\zeta_2)\le \delta\quad\mbox{for all}\quad\xi\le-A.
\eea


For any fixed 
$T\in\mathbb{R}$,
we denote
\beaa
w_u^{\sigma}(t,\xi)=u^{\infty}(t+T,\xi+\sigma),\quad w_v^{\sigma}(t,\xi)=v^{\infty}(t+T,\xi+\sigma)
\eeaa
for all $\sigma\in\mathbb{R}$ and $(t,\xi)\in\mathbb{R}\times\mathbb{R}$. Define now
$\sigma^*=\inf \mathcal{A}$, where
\begin{equation*}
\mathcal{A}:=\{\sigma\in\mathbb{R}|\ w_u^{\sigma'}\le u^{\infty},\  w_v^{\sigma'}\ge v^{\infty} \ \mbox{in}\ \ \mathbb{R}\times\mathbb{R}\ \  \mbox{for all}\ \ \sigma'\ge\sigma\}.
\end{equation*}
Lemma \ref{lm: lower and upper esitimates of tilde u v} implies that $w^{\sigma}_u\le u^{\infty}$
and $w^{\sigma}_v\ge v^{\infty}$ for $(t,\xi)\in\mathbb{R}\times\mathbb{R}$ and for all $\sigma\ge2\zeta_2$. Thus,
$\mathcal{A}$ is non-empty. Moreover, since $(U,V)(-\infty)=(1,0)$
and $(U,V)(+\infty)=(0,1)$, we see that $\mathcal{A}$ is bounded from below. Thus,
$\sigma^*$ is well defined and is finite. Moreover, by continuity, we have
\bea\label{cont}
w_u^{\sigma^*}\le u^{\infty},\quad w_v^{\sigma^*}\ge v^{\infty}\quad\mbox{for all}\quad(t,\xi)\in\mathbb{R}\times\mathbb{R}.
\eea

Define $E_1:=\{(t,\xi)\in\mathbb{R}\times[-A,\infty)\}$ and $E_2:=\{(t,\xi)\in\mathbb{R}\times(-\infty,-A]\}$. We now prove the following key result:

\begin{claim}\label{cl:1}
There exists no $\eta_0>0$ such that
$$w_u^{\sigma^*-\eta}\le u^{\infty}\ \ \mbox{and}\  \  w_v^{\sigma^*-\eta}\ge v^{\infty}\ \
\mbox{in}\ \ E_1\quad\text{for all}\quad \eta\in[0,\eta_0].$$
\end{claim}
\begin{proof}
Assume that such a $\eta_0$ exists. We shall show that it would also hold
\bea\label{goal-cl:1}
w_u^{\sigma^*-\eta}\le u^{\infty}\ \ \mbox{and}\ \ w_v^{\sigma^*-\eta}\ge v^{\infty}\ \ \mbox{in}\ \ E_2\quad\text{for all}\quad \eta\in[0,\eta_0].
\eea
Define
\beaa
&&\varepsilon_u^*=\inf\{\varepsilon>0|\, u^{\infty}+\varepsilon\ge w_u^{\sigma^*-\eta}\ \  \mbox{for all}\ \ (t,\xi)\in E_2\},\\
&&\varepsilon_v^*=\inf\{\varepsilon>0|\, v^{\infty}-\varepsilon\le w_v^{\sigma^*-\eta}\ \  \mbox{for all}\ \ (t,\xi)\in E_2\}.
\eeaa
Then the real numbers $\varepsilon_u^*$ and $\varepsilon_v^*$ are nonnegative.
To show that $\varepsilon_u^*=\varepsilon_v^*=0$,
we first assume $\varepsilon_u^*\ge\varepsilon_v^* >0$. Since $w_u^{\sigma^*-\eta}\le u^{\infty}$ for $\xi=-A$, there exist sequences $\{\xi_n\}$ which converges to $\xi_{\infty}\in(-\infty,-A)\cup\{-\infty\}$ and $\{t_n\}\subset\mathbb{R}$ such that
\bea\label{time-space-shift}
u^{\infty}(t_n,\xi_n)+\varepsilon_u^*-w_u^{\sigma^*-\eta}(t_n,\xi_n)\to 0\ \ \mbox{as $n\to\infty$\ \ and }\ \
v^{\infty}(t_n,\xi_n)-\varepsilon_u^*\leq w_v^{\sigma^*-\eta}.
\eea

Since $U(-\infty)=1$ and \eqref{u v infty betweent two tw}, we assert that $\xi_{\infty}\neq -\infty$.
Set
\beaa
(u^{\infty}_n,v^{\infty}_n)(t,\xi):=(u^{\infty},v^{\infty})(t+t_n,\xi),\quad
(w_{u,n}^{\sigma^*-\eta},w_{v,n}^{\sigma^*-\eta})(t,\xi):=(w_u^{\sigma^*-\eta},w_v^{\sigma^*-\eta})(t+t_n,\xi).
\eeaa
Then from the standard parabolic estimates, $(u^{\infty}_n,v^{\infty}_n)$ and $(w_{u,n}^{\sigma^*-\eta},w_{v,n}^{\sigma^*-\eta})$
converge locally uniformly, up to extraction of a subsequence, to a solution $(\bar{u}^{\infty},\bar{v}^{\infty})$ of \eqref{limit system moving frame} and $(\bar{w}_{u}^{\sigma^*-\eta},\bar{w}_{v}^{\sigma^*-\eta})$, respectively, such that
\begin{equation*}
z(t,\xi):=\bar{u}^{\infty}(t,\xi)+\varepsilon_u^*-\bar{w}_{u}^{\sigma^*-\eta}(t,\xi)\ge 0\quad\mbox{for all}\quad (t,x)\in E_2.
\end{equation*}
Moreover, due to \eqref{time-space-shift}, we have $z(0,\xi_{\infty})=0$ and
\bea\label{v-infty-wv}
\bar{v}^{\infty}-\varepsilon_u^*\leq \bar{w}_{v}^{\sigma^*-\eta}\quad\mbox{for all}\quad (t,\xi)\in E_2.
\eea

Recall $F$ from \eqref{FG}.
Since $\delta>0$ is chosen small enough, it follows that $F(u,v)$ is decreasing in both $u$ and $v$ for
$(u,v)\in D:=\{1-\delta\leq u \leq 1, 0\leq v\leq \delta\}$. Also,
note that, for $\xi\le -A$, we have $(\bar{u}^{\infty},\bar{v}^{\infty})\in D$.
Using $a<1$,
it follows from some straightforward computation that
\beaa
\partial_tz-z_{\xi\xi}-c^*z_{\xi}&=&F(\bar{u}^{\infty},\bar{v}^{\infty})
-F(\bar{w}_{u}^{\sigma^*-\eta},\bar{w}_{v}^{\sigma^*-\eta})\\
&\geq& F(\bar{u}^{\infty}+\varepsilon_u^*,\bar{v}^{\infty}-\varepsilon_u^*)
-F(\bar{w}_{u}^{\sigma^*-\eta},\bar{w}_{v}^{\sigma^*-\eta}).
\eeaa
By \eqref{v-infty-wv}, the Lipschitz continuity, and monotonicity of $F(u,v)$ in $u$, there exists $C_\delta>0$ such that
\beaa
\partial_tz-z_{\xi\xi}-c^*z_{\xi}
\geq F(\bar{u}^{\infty}+\varepsilon_u^*,\bar{w}_{v}^{\sigma^*-\eta})-F(\bar{w}_{u}^{\sigma^*-\eta},\bar{w}_{v}^{\sigma^*-\eta})\geq -C_{\delta}z
\eeaa
for all $\xi\le-A$. Since $z(0,\xi_{\infty})=0$,
the strong maximum principle implies $z(t,\xi)=0$ for all $t\le 0$ and $\xi\le-A$. However, this is contradict to $z(t,-A)=\varepsilon_u^*>0$. Therefore, $\varepsilon_u^*\ge\varepsilon_v^*>0$ is impossible.

If $\varepsilon_v^*\ge\varepsilon_u^*>0$, by repeating the similar argument and considering the function
$$z(t,\xi):=\bar{w}_{v}^{\sigma^*-\eta}(t,\xi)-\varepsilon_v^*-\bar{v}^{\infty}(t,\xi),$$
we also can prove that $\varepsilon_v^*\ge\varepsilon_u^*>0$ is impossible. Therefore, we assert that $\varepsilon_u^*=\varepsilon_v^*=0$ and
\eqref{goal-cl:1} holds.
It follows that $\sigma^*-\eta_0\in \mathcal{A}$, which contradicts to the definition of $\sigma^*$. Therefore, we complete the proof of Claim \ref{cl:1}.
\end{proof}

We next show that $\sigma^*\leq0$. For contradiction, we assume that $\sigma^*>0$.
From Claim \ref{cl:1}, there exist two sequences $\{\sigma_n\}_{n\in\mathbb{N}}$ in $(0,\sigma^*)$ and $\{(\tau_n,\xi_n)\}_{n\in\mathbb{N}}\subset \{\xi\geq -A\}$ such that
$\sigma_n\to\sigma^*$ as $n\to+\infty$, and it holds that
\bea\label{wu ge u-infty}
w_u^{\sigma_n}(\tau_n,\xi_n)\ge u^{\infty}(\tau_n,\xi_n)\quad \mbox{for all}\quad n\in\mathbb{N}
\eea
or
\beaa
w_v^{\sigma_n}(\tau_n,\xi_n)\le v^{\infty}(\tau_n,\xi_n)\quad \mbox{for all}\quad n\in\mathbb{N}.
\eeaa

\begin{claim}\label{cl:2}
$\{\xi_n\}$ must be bounded.
\end{claim}
\begin{proof}
Note that $\xi_n\ge-A$ for all $n\in\mathbb{N}$. Hence, if the result is not true, up to extracting a subsequence we may assume
$\xi_{n}\to+\infty$ as $n\to\infty$.
We choose $N>0$ large enough such that, for each $n\ge N$, we have
$\sigma^*-\sigma_n<1$ and
$\xi_n\ge A$ and $u^{\infty}(\tau_n,\xi_n)\in[0,\delta]$ because of
\eqref{u v infty betweent two tw} and \eqref{U-V-right plane}. Moreover, from Lemma \ref{lm: behavior around + infty} and \eqref{u v infty betweent two tw},
there exists a constant $C_1>0$ such that, for any $t_0\in\mathbb{R}$ and $\xi_0\ge -A+2$, it holds
\bea\label{Harnack}
\max_{t_0-1\le t\le t_0,\ |\xi-\xi_0|\le 2}u^{\infty}(t,\xi)\le \min_{t_0-1\le t\le t_0,\ |\xi-\xi_0|\le 2}C_1u^{\infty}(t,\xi).
\eea
Using \eqref{wu ge u-infty}, \eqref{Harnack} and the standard parabolic estimates, there exists $C_2>0$ such that, for $n\ge N$, it holds
\begin{equation*}
\begin{aligned}
0\le&u^{\infty}(\tau_n,\xi_n)-u^{\infty}(\tau_n+T,\xi_n+\sigma^*)\\
\le&w_u^{\sigma_n}(\tau_n,\xi_n)-w_u^{\sigma^*}(\tau_n,\xi_n)\\
\le&C_2(\sigma^*-\sigma_n)\Big[\max_{\tau_n-1\le t\le \tau_n,\ |\xi-\xi_n|\le 2}w_u^{\sigma^*}(t,\xi)\Big]\\
\le&C_1C_2(\sigma^*-\sigma_n) w_u^{\sigma^*}(\tau_n,\xi_n)\\
\le &C_1C_2(\sigma^*-\sigma_n) U(\xi_n+\sigma^*-\zeta_1),
\end{aligned}
\end{equation*}
where the last inequality follows from \eqref{u v infty betweent two tw}.

Now, let us first assume $T>0$. Then, from the regularity of $u^{\infty}$, Lemma \ref{lm: behavior around + infty} and \eqref{u v infty betweent two tw}, there also exists $C_3>0$ such that
$$(u^{\infty}-w_u^{\sigma^*})(t-T,\xi-\sigma^*)\le C_3(u^{\infty}-w_u^{\sigma^*})(t,\xi)\quad\mbox{for all}\quad(t,\xi)\in\mathbb{R}\times\mathbb{R}.$$
Thus, we have
$$u^{\infty}(\tau_n-kT,\xi_n-k\sigma^*)-u^{\infty}(\tau_n-(k-1)T,\xi_n-(k-1)\sigma^*)\le C_1C_2C_3^k(\sigma^*-\sigma_n)U(\xi_n+\sigma^*-\zeta_1),$$
for all $k\in\mathbb{N}$ and $n\ge N$, whence
$$u^{\infty}(\tau_n-kT,\xi_n-k\sigma^*)-u^{\infty}(\tau_n+T,\xi+\sigma^*)\le C_1C_2\Big[\sum_{i=0}^kC_3^i\Big](\sigma^*-\sigma_n) U(\xi_n+\sigma^*-\zeta_1).$$
From \eqref{u v infty betweent two tw}, we have
\begin{equation}\label{1}
U(\xi_n+\zeta_2-k\sigma^*)\le \left[1+C_1C_2\sum_{i=0}^kC_3^i(\sigma^*-\sigma_n)\right] U(\xi_n+\sigma^*-\zeta_2),
\end{equation}
for all $k\in\mathbb{N}$ and $n\ge N$. Since $\sigma^*>0$, we can find $k$ such that $-k\sigma^*<\sigma^*-2\zeta_2$. From Lemma \ref{lm: behavior around + infty}, there exists $C_4>0$ such that
$U(s+\zeta_2-k\sigma^*)\ge(1+C_4)U(s+\sigma^*-\zeta_2)$ for all $s$ sufficiently large. Since $U>0$ and $\xi_n\to+\infty$, $\sigma_n\to\sigma^*$ as $n\to+\infty$, \eqref{1} is impossible to hold for large $n$. For the case $T<0$, we can get a contradiction by applying the same argument.
Therefore, we complete the proof of Claim~\ref{cl:2}.
\end{proof}

\medskip

Due to Claim~\ref{cl:2}, up to extraction of a subsequence we may assume that $\xi_{n}\to \xi_*\in[-A,+\infty)$ as $n\to\infty$.
Consider $(u^{\infty}_n,v^{\infty}_n)(t,\xi):=(u^{\infty},v^{\infty})(t+\tau_n,\xi)$.
By standard parabolic estimates, up to extraction of a subsequence, we have
that $(u^{\infty}_n,v^{\infty}_n)$
converge locally uniformly in $\mathbb{R}\times\mathbb{R}$ to a solution $(\bar{u}^{\infty},\bar{v}^{\infty})$ of \eqref{limit system moving frame}. Furthermore, in view of \eqref{cont}, we have
\beaa
&&z_u(t,\xi)=\bar{u}^{\infty}(t,\xi)-\bar{u}^{\infty}(t+T,\xi+\sigma^*)\ge 0\ \ \mbox{in}\ \ \mathbb{R}\times\mathbb{R},\\
&&z_v(t,\xi)=\bar{v}^{\infty}(t,\xi)-\bar{v}^{\infty}(t+T,\xi+\sigma^*)\le 0\ \ \mbox{in}\ \ \mathbb{R}\times\mathbb{R}.
\eeaa
Note that $z_u(0,\xi_*)=0$. Then the strong maximum principle and  uniqueness of solutions to the Cauchy problem for \eqref{limit system moving frame} imply that $z_u\equiv 0$ in $\mathbb{R}\times\mathbb{R}$, whence $\bar{u}^{\infty}(t,\xi)=\bar{u}^{\infty}(t+T,\xi+\sigma^*)$ in $\mathbb{R}\times\mathbb{R}$.
In particular, $\bar{u}^{\infty}(0,0)=\bar{u}^{\infty}(jT,j\sigma^*)$ for all $j\in\mathbb{Z}$.
However,
thanks to \eqref{u v infty betweent two tw}, we see that
\beaa
U(\xi+\zeta_2)\le \bar{u}(t,\xi)\le U(\xi-\zeta_2)\quad\mbox{for all}\quad(t,\xi)\in\mathbb{R}\times \mathbb{R}.
\eeaa
Since $\sigma^*>0$ (here we actually use $\sigma^*\neq0$) and $U(-\infty)=1>0=U(+\infty)$, we have reached a contradiction.

From the above discussions, we have proved that $\sigma^*\le 0$. Thus, for all $\sigma\ge 0$, we have
$$u^{\infty}(t,\xi)\ge w_u^{\sigma}(t,\xi)=u^{\infty}(t+T,\xi+\sigma),\quad v^{\infty}(t,\xi)\le w_v^{\sigma}(t,\xi)=v^{\infty}(t+T,\xi+\sigma),$$
for all $(t,\xi)\in\mathbb{R}\times\mathbb{R}$. Furthermore, since $T\neq 0$ can be chosen arbitrarily,
it follows that
\bea\label{to be TW}
u^{\infty}(t,\xi)=\phi_u(\xi),\quad v^{\infty}(t,\xi)=\phi_v(\xi)\quad \mbox{for all}\quad (t,\xi)\in\mathbb{R}\times \mathbb{R},
\eea
for some nonincreasing function $\phi_u(\xi)$ and nondecreasing function $\phi_v(\xi)$.
On the other hand, the strong maximum principle implies that it holds either
\beaa
u^{\infty}(t,\xi)>u^{\infty}(t+T,\xi+\sigma) \quad \mbox{(resp., $v^{\infty}(t,\xi)<v^{\infty}(t+T,\xi+\sigma)$\, )}
\eeaa
or
\bea\label{identity}
u^{\infty}(t,\xi)\equiv u^{\infty}(t+T,\xi+\sigma)\quad  \mbox{(resp., $v^{\infty}(t,\xi)\equiv v^{\infty}(t+T,\xi+\sigma)$\, )}.
\eea
By \eqref{u v infty betweent two tw},
$(\phi_u(\xi),\phi_v(\xi))$ satisfies
$(\phi_u(-\infty),\phi_v(-\infty))=(1,0)$ and $(\phi_u(+\infty),\phi_v(+\infty))=(0,1)$.
Thus, \eqref{identity} 
is impossible; so we assert that both $\phi_u(\cdot)$ and $\phi_v(\cdot)$ are strictly monotone functions. 
Therefore,  $(\phi_u(\cdot), \phi_v(\cdot))$ forms a strictly monotone traveling wave solution,
and is trapped between two shifts of minimal traveling waves (due to \eqref{u v infty betweent two tw}).
The standard sliding method (see, e.g., \cite[Proposition A.7]{Girardin Lam}) yields that for some $\tilde{\zeta}$,
\bea\label{To be TW-shift}
\phi_u(\xi)=U(\xi-\tilde{\zeta}),\quad \phi_v(\xi)=V(\xi-\tilde{\zeta})\quad \mbox{for all}\quad\xi\in\mathbb{R}.
\eea

Combining \eqref{to be TW} and \eqref{To be TW-shift}, we complete the proof of Lemma~\ref{u v infity converges to tw}.
\end{proof}

\begin{remark}
The uniqueness (up to translations) of traveling wave solutions for \eqref{system} under {\bf(H1)} is not completely solved.
It was proved in \cite[Corollary A.7]{Girardin Lam} that if $d\leq 2+\frac{r}{1-a}$, the traveling wave solution is unique
(up to translations). In the proof of \eqref{To be TW-shift}, we do not need the above restrictions on parameters because
$(\phi_u(\cdot), \phi_v(\cdot))$ is trapped between two shifts of minimal traveling waves, which guarantees that
the standard sliding method works.
\end{remark}

We are ready to prove Theorem~\ref{th: profile}.

\begin{proof}[Proof of Theorem~\ref{th: profile}]
Recall $(\tilde{u},\tilde{v})$ from \eqref{def of tilde u v}.
Thanks to Lemma~\ref{u v infity converges to tw}, we have
\beaa
\lim_{n\to\infty}(\tilde{u},\tilde{v})(t+t_n,\xi)= (U,V)(\xi-\tilde{\zeta})\quad \mbox{in}\quad C_{loc}^{1,2}(\mathbb{R}\times\mathbb{R}).
\eeaa
Since the time sequence $\{t_n\}$ can be chosen arbitrarily and $\tilde{\zeta}$ is independent of the choice of $\{t_n\}$ (due to Lemma~\ref{lem:asymptotic stability of tw}), we have
\bea\label{local conv}
\lim_{t\to\infty}(\tilde{u},\tilde{v})(t,\xi)=(U,V)(\xi-\tilde{\zeta})\quad \mbox{locally uniformly for $\xi\in\mathbb{R}$}.
\eea
Moreover, in view of Lemma~\ref{lower and upper esitimates of tilde u v} and the fact that $(U,V)(-\infty)=(1,0)$ and $(U,V)(\infty)=(0,1)$,
we see that for any given $\epsilon>0$, there exists
$T>0$ and $M>0$ such that
\beaa
|(\tilde{u},\tilde{v})(t,\xi)- (U,V)(\xi-\tilde{\zeta})|<\epsilon,
\eeaa
provided $t\geq T$ and $\xi\in\{-c^*t \leq \xi\leq -M\}\cup\{\xi\geq M\}$,
which, combined with \eqref{local conv}, yields Theorem~\ref{th: profile}. This completes the proof.
\end{proof}

\section{Cauchy problem with scenario \eqref{initial data2}}
\label{Sec:Cauchy prob-ic-2}
\setcounter{equation}{0}

In this section, we shall consider the initial data that satisfies \eqref{initial data2}
and prove Theorem~\ref{th2: profile}.


\begin{proof}[Proof of Theorem~\ref{th2: profile}]
Let $(u,v)$ be the solution of system \eqref{system} with initial data $(u_0,v_0)$ satisfying \eqref{initial data2}.
We first show that
\bea\label{th2-goal-1}
&&\lim_{t\to\infty}\left[\sup_{x\geq c_0t}
\Big|u(t,x)-U(x-c^* t-h_1)\Big|+\sup_{x\in[0,c_0t)}\Big|v(t,x)-V(x-c^* t-h_1)\Big|
\right]=0,
\eea
where $h_1$ is a constant and $c_0\in(c^*,c_v)$.
To do so, let us consider $(\tilde{u}_0,\tilde{v}_0)$ satisfying
\bea\label{th2-proof-new-ic}
\tilde{u}_0(x)=u_0(x),\quad \tilde{v}_0(x)\ge v_0(x)\quad \mbox{and $\quad \tilde{v}_0(x)\geq \delta>0$\ \ for some \ \ $\delta>0$}.
\eea
Let $(\tilde{u},\tilde{v})$ be the solution of system \eqref{system} with the initial data  $(\tilde{u}_0,\tilde{v}_0)$
satisfying \eqref{th2-proof-new-ic}.
Then, by applying comparison principle, we have
\bea\label{th2-cp-result-1}
\tilde{u}(t,x)\leq u(t,x),\quad \tilde{v}(t,x)\geq v(t,x)\quad \mbox{for all}\quad (t,x)\in\mathbb{R}_+\times\mathbb{R}.
\eea
Now, we define a sub-solution $(\underline{u},\bar{v})$ as \eqref{def of  sub sol no ac} for $x\ge 0$.
Note that $(\tilde{u}_0,\tilde{v}_0)$ satisfies \eqref{initial data}. Thus, we can
choose suitable parameters in $(\underline{u},\bar{v})$, and
use Lemma~\ref{lm: sub sol boundary condition no ac} to conclude that, for some large $T^*$,
\bea\label{th2-cp-result-2}
\tilde{u}(t+T^*,x)\geq \underline{u}(t,x),\quad \tilde{v}(t+T^*,x)\leq \bar{v}(t,x)
\quad \mbox{for all}\quad (t,x)\in[T_0,\infty)\times\mathbb{R}_+.
\eea
By \eqref{th2-cp-result-1} and \eqref{th2-cp-result-2}, we obtain that
\bea\label{th2-cp-result-3}
u(t+T^*,x)\geq \underline{u}(t,x),\quad v(t+T^*,x)\leq \bar{v}(t,x)
\quad \mbox{for all}\quad (t,x)\in[T_0,\infty)\times\mathbb{R}_+.
\eea

Next, let us define a super-solution $(\bar{u},\underline{v})$ as \eqref{def of super sol} for $x\ge 0$.
We now compare $(u,v)$ with $(\bar{u},\underline{v})$ over $\Omega_{T}(t)$ for some large $T$ and  $c_0\in(c^*,c_v)$,
where
\beaa
\Omega_T(t):=\{(x,t)|\,t\geq T,\ 0\leq x\leq c_0t\}.
\eeaa
Let us focus on $\{x=c_0t\}$ first.
From the definition of $\bar{u}$ and Lemma~\ref{lm: decay estimamte u and v}(i)
\beaa
\bar{u}(t,c_0t)-u(t,c_0t)&\geq& U(x-c^*t-\zeta(t))+P(t)e^{-\alpha_1((c_0-c^*)t+x_1)}-C_1e^{-\nu_1 t}\\
                       &\geq& p_1 e^{-\mu_1 t}e^{-\alpha_1((c_0-c^*)t+x_1)}-C_1e^{-\nu_1 t}
\eeaa
We choose $\mu_1>0$ small enough and $c_0>c^*$ sufficiently close to $c^*$ such that
\bea\label{th2:choice of mu}
\mu_1+\alpha_1(c_0-c^*)<\nu_1.
\eea
Note that from the proof of Lemma~\ref{lm: super sol inequality}, we see that
the choice of $T_0$ is independent on all small $\mu$.
Therefore, there exists some $T_1>T_0$ such that
\bea\label{th2-rbc for u}
\bar{u}(t,c_0t)-u(t,c_0t)\geq0\quad\text{for all}\quad t\geq T_1.
\eea
Also, from the definition of $\underline{v}$, and Lemma~\ref{lm: decay estimamte u and v}(ii), we have
\beaa
v(t,c_0t)-\underline{v}(t,c_0t)\geq 1-C_2e^{-\nu_2 t}-V(x-c^*t-\zeta(t))+Q(t)
                       \geq -C_2e^{-\nu_2 t}+q_1 e^{-\mu_1 t}.
\eeaa
By setting $\mu_1<\nu_2$,
there exists  $T_2>T_0$ such that
\bea\label{th2-rbc for v}
v(t,c_0t)-\underline{v}(t,c_0t)\geq 0\quad\text{for all}\quad t\geq T_2.
\eea

Next, we consider the left boundary $\{x=0\}$.
From the definition of $\bar{u}$ and Lemma~\ref{lm: decay estimamte u and v}(i), we have
\beaa
\bar{u}(t,0)-u(t,0)&\geq& U(-c^*t-\zeta(t))+P(t)e^{-\alpha_1(-c^*t+x_1)}-C_1e^{-\nu_1 t}\\
                       &\geq& p_1 e^{-\mu_1 t}\min\{e^{\alpha_1 c^*t-\alpha_1  x_1},1\}-C_1e^{-\nu_1 t}.
\eeaa
For all $t\gg1$, $\min\{e^{\alpha c^*t-\alpha  x_1},1\}=1$.
Hence, since \eqref{th2:choice of mu}, there exists $T_3>T_0$ such that
\bea\label{th2-lbc for u}
\bar{u}(t,0)-u(t,0)\geq 0\quad \text{for all}\quad t\geq T_3.
\eea
Similarly, by applying $\mu_1<\nu_2$, we can assert that for some $T_4>T_0$,
\bea\label{th2-lbc for v}
v(t,0)-\underline{v}(t,0)\geq 0\quad\text{for all}\quad t\geq T_4.
\eea

Let us fix $T:=\max\{T_1,T_2,T_3,T_4\}$. If necessary, we may shift $(\bar{u},\underline{v})$ (setting $\zeta_0$ sufficiently large does not affect $T_i$ and $T$) such that $\bar{u}(T,\cdot)\geq u(T,\cdot)$ and
$\underline{v}(T,\cdot)\leq v(T,\cdot)$. Together with
the conclusion of Lemma~\ref{lm: super sol inequality} and \eqref{th2-rbc for u}, \eqref{th2-rbc for v},
\eqref{th2-lbc for u} and \eqref{th2-lbc for v}, we can apply the comparison principle to conclude that
for some $T^{**}$, it holds
\bea\label{th2-cp-result-4}
{u}(t,x)\leq \bar{u}(t+T^{**},x),\quad {v}(t,x)\geq \underline{v}(t+T^{**},x)
\quad \mbox{for all}\quad (t,x)\in[T,\infty)\times[0,c_0t],
\eea
for some $c_0\in(c^*,c_v)$ with sufficiently close to $c^*$.
Combining \eqref{th2-cp-result-3} and \eqref{th2-cp-result-4}, we can follow the same line as in the proof of Theorem~\ref{th: profile} to obtain \eqref{th2-goal-1}.

Finally, by Lemma~\ref{lm: decay estimamte u and v}(i),
we can follow the argument
of \cite[Section 4.1]{Peng Wu Zhou} that modified the argument of \cite{log delay 3} to conclude
\beaa
&&\lim_{t\to\infty}\left[\sup_{x\in[ct,\infty)}\Big|v(t,x)-V_{KPP}(x-c_{v} t+\frac{3d}{c_v}\ln t+\omega(t))\Big|+\sup_{x\in[ct,\infty)}|u(t,x)|\right]=0,
\eeaa
from which the proof of Theorem~\ref{th2: profile} is complete.
\end{proof}

\bigskip

\noindent{\bf Acknowledgement.}
Maolin Zhou is supported by the National Key Research and Development Program of China (2021YFA1002400).
Chang-Hong Wu is supported by the Ministry of Science and Technology of Taiwan.
Dongyuan Xiao is supported by the LabEx {\it Solutions Num\'{e}riques, Mat\'{e}rielles et Mod\'{e}lisation pour lEnvironnement et le Vivant} (NUMEV) of the University of Montpellier.

\end{document}